\documentclass[11pt, oneside]{amsart}
\usepackage{wrapfig}
\usepackage{tikz}
\usepackage{comment}
\usepackage{mathrsfs}
\usepackage{tabularx, hyperref}
\usepackage{amssymb} \usepackage{amsfonts}
 \usepackage{amsmath}
\usepackage{amsthm} \usepackage{epsfig, subfig}
\usepackage{ amscd, amsxtra, latexsym}
\usepackage{epsfig,  graphicx, psfrag}
\usepackage[all]{xy}
\usepackage{caption}
\usepackage{enumerate}
\usepackage{color}

\DeclareMathOperator\supp{Supp}
\DeclareMathOperator\lip{Lip}

\DeclareMathOperator\isom{Isom}
\DeclareMathOperator\Def{def}

\newtheorem{lemma}{Lemma}[section]
\newtheorem{thm}[lemma]{Theorem}
\newtheorem{prop}[lemma]{Proposition}
\newtheorem{cor}[lemma]{Corollary}

\newtheorem*{prop*}{Proposition}
\newtheorem{prop_intro}{Proposition}
\newtheorem{thm_intro}[prop_intro]{Theorem}

\newtheorem{quest_intro}[prop_intro]{Question}
\theoremstyle{definition}
\newtheorem{defn_intro}[prop_intro]{Definition}
\newtheorem{defn}[lemma]{Definition}

\newtheorem{quest}[lemma]{Question}

\newtheorem{rem}[lemma]{Remark}

\theoremstyle{definition}
\newtheorem{remark}[lemma]{Remark}

\definecolor{darkgreen}{cmyk}{1,0,1,.2}

\newcommand{\g} {\ensuremath {\gamma}}

\DeclareMathOperator{\alt}{alt}

\newcommand{\QCa}{{\rm QZ}_{\alt}}
\newcommand{\QC}{{\rm QZ}}

\renewcommand{\:}{\,:}
\newcommand{\Cay}{\ensuremath{{\rm Cay}}}
\newcommand{\actson}{\ensuremath{\curvearrowright}}
\newcommand\red{\textrm{red}}

\newcommand{\ov}{\ensuremath {\overline}}
\newcommand{\N}{\ensuremath {\mathbb{N}}}
\newcommand{\R} {\ensuremath {\mathbb{R}}}

\newcommand{\Z} {\ensuremath {\mathbb{Z}}}
\newcommand{\G} {\ensuremath {\Gamma}}

\newcommand{\calC} {\ensuremath {\mathcal{C}}}

\newcommand{\calH} {\ensuremath {\mathcal{H}}}

\newcommand{\calL}{\ensuremath {\mathcal{L}}}

\newcommand{\calX} {\ensuremath {\mathcal{X}}}

\newcommand{\calR} {\ensuremath {\mathcal{R}}}
\newcommand{\bah} {\ensuremath {{\infty,0}}}
\renewcommand{\phi}{\varphi}

\newcommand{\cal}{\mathcal}

\newcommand{\Isom}{\ensuremath{{\rm Isom}}}

\newcommand{\wt}{\widetilde}
\newcommand{\vol}{{\rm vol}}
\renewcommand{\H}{\mathbb H}
\renewcommand{\S}{\Sigma}

\begin{document}

\title[The zero norm subspace for acylindrically hyperbolic groups]{The zero norm subspace of bounded cohomology of acylindrically hyperbolic groups}

\author[]{F. Franceschini}
\address{Karlsruher Institut f\"ur Technologie (KIT), Fakult\"at f\"ur Mathematik, Institut f\"ur Algebra und Geometrie Englerstra\ss e. 2, 76131 Karlsruhe, Deutschland}
\email{federico.franceschini@kit.edu}

\author[]{R. Frigerio}
\address{Dipartimento di Matematica, Universit\`a di Pisa, Largo B. Pontecorvo 5, 56127 Pisa, Italy}
\email{frigerio@dm.unipi.it}

\author[]{M. B. Pozzetti}
\address{Mathematics Institute, Zeeman Building, University of Warwick, Coventry CV4 7AL United Kingdom}
\email{B.Pozzetti@Warwick.ac.uk}

\author[]{A. Sisto}
\address{Department Mathematik, ETH Z\"urich, 
R\"amistrasse 101, CH-8092 Z\"urich, Switzerland}
\email{sisto@math.ethz.ch}
\date{\today}
\thanks{}

\keywords{}
\begin{abstract}
We construct combinatorial volume forms of hyperbolic three manifolds fibering over the circle. These forms define non-trivial classes in bounded cohomology.  
After introducing a new seminorm on exact bounded cohomology, we use these combinatorial classes to show that,  
in degree 3, the zero norm subspace of the  bounded cohomology of an acylindrically hyperbolic group is infinite dimensional. 
In the appendix we use the same techniques to give a cohomological proof of a lower bound, originally due to Brock, on the volume of the mapping torus of a cobounded pseudo-Anosov homeomorphism of a closed surface in terms of its Teichm\"uller translation distance.
\end{abstract}
\maketitle

Even for very well-studied groups such as non-abelian free groups, the task of computing the bounded cohomology in higher degrees is  still challenging. 
In degree 2, the technology of quasimorphisms has been extensively  exploited to construct non-trivial bounded cohomology
classes (see e.g.~\cite{Brooks, EpsteinFuji, Fujiwara1,BeFu-wpd, FujiTAMS} for the case of trivial coefficients, 
and~\cite{HullOsin,BeBrFu,CFI} for more general coefficient modules). On the other hand, in higher degrees both constructing bounded cocycles and showing that such cocycles define non-trivial bounded cohomology classes is 
definitely non-trivial. For example, as far as the authors know, 
in the case of non-abelian free groups, non-trivial bounded classes in degree 3 have been constructed only with the help of hyperbolic geometry (see e.g.~\cite{Soma3,Soma5, Soma1}), and it is still a major open question whether
the fourth bounded cohomology of non-abelian free groups vanishes or not.

The purpose of this paper is twofold. First, we construct a discrete $3$-dimensional volume form
on a class of free-by-cyclic groups.
Then, building on results from~\cite{FPS}, we exploit our construction to show that, for every acylindrically hyperbolic group, the space of bounded classes with vanishing seminorm is infinite dimensional in degree 3.

Following a suggestion by Mladen Bestvina, our construction is based on a suitable relative version of Mineyev's bicombing on hyperbolic groups~\cite{Mineyev1},
which is due to Groves and Manning~\cite{GM} and Franceschini~\cite{Franceschini2}. Dealing with a discrete volume form rather than with differential forms allows us to provide a somewhat unified version of the 
arguments developed in~\cite{Soma1},
where some essential estimates make use of a careful comparison between the volume forms arising from the hyperbolic and the singular Sol structure supported by hyperbolic $3$-manifolds that fiber over the circle. 
We hope that our combinatorial arguments, although clearly inspired by their differential counterpart, could be more easily extended to wider classes of groups and, maybe, even to higher degrees.

\subsection*{Bounded cohomology of discrete groups}
Let $\G$ be a group. 
We briefly recall the definition of bounded cohomology of $\G$ (with trivial real coefficients), referring the reader to Section~\ref{background:sec} for more details.
We denote by $C^n(\G)$ the set of real-valued homogeneous $n$-cochains on $\G$,
and for every $\varphi\in C^n(\G)$
we set
$$
\|\varphi\|_\infty=\sup \{|\varphi (g_0,\ldots,g_n)|\, |\, (g_0,\ldots,g_n)\in \G^{n+1}\}\ \in \ [0,\infty]\ .
$$
We denote by $C^n_b(\G)\subseteq C^n(\G)$ the subspace
of bounded cochains, and by $C^n(\G)^\G$, $C^n_b(\G)^\G$ the subspaces of invariant (bounded) cochains.
The cohomology of the complex
$C^*_b(\G)^\G$ is the
\emph{bounded cohomology} $H_b^*(\G)$ of $\G$.
The norm $\|\cdot\|_\infty$ on $C^n_b(\G)$ induces a seminorm on $H^n_b(\G)$ that is usually called the \emph{Gromov seminorm}.

The inclusion of (invariant) bounded cochains into ordinary cochains induces the \emph{comparison map} 
$c^n\colon H^n_b(\G)\to H^n(\G)$. The kernel of $c^n$ is the set of bounded cohomology classes whose representatives are
exact, and it is denoted by $EH^n_b(\G)$. 
By definition, a class $\alpha \in EH^{n+1}_b(\G)$ is represented
by a bounded cocycle $z=\delta \varphi\in C^{n+1}_b(\G)^\G$, where $\varphi\in C^{n}(\G)^\G$ is a (possibly unbounded) 
cochain. 
In other words, if we define the space $\QC^n(\G)\subseteq C^n(\G)$ of $n$-quasi-cocycles as the 
subset of cochains having bounded differential, 
then the differential induces a surjection 
$\QC^n(\G)^\G\longrightarrow EH^{n+1}_b(\G)$.

We denote by $N^n(\G)$ the subspace of $H^n_b(\G)$ given by elements with vanishing Gromov seminorm. 
It is easy to show that $N^n(\G)\subseteq EH^n_b(\G)$ for every $n\in\mathbb{N}$ (see Lemma~\ref{lem:Nexact}). 
It was proved by Matsumoto and Morita~\cite{Matsu-Mor}  and independently by Ivanov~\cite{Ivanov2} that
$N^2(\G)=0$ for every group $\G$. On the other hand, Soma proved that $N^3(F_2)\neq 0$~\cite{Soma2}, and that the dimension of $N^3(\Gamma_g)$ has the cardinality of the continuum~\cite{Soma1}, where
$F_2$ and $\Gamma_g$ denote respectively the free group on two generators and the fundamental group of a closed orientable surface of genus $g\geq 2$.

\subsection*{Main results}
In this paper we extend Soma's results as follows:

\begin{thm_intro}\label{generalcase}
Suppose that $\G$ is acylindrically hyperbolic. Then the dimension of $N^3(\G)$ has the cardinality of the continuum.
\end{thm_intro}

A group $\G$ is \emph{acylindrically hyperbolic} if it admits a non-elementary acylindrical action on
a Gromov hyperbolic space~\cite{Os-acyl}. The class of acylindrically hyperbolic groups includes many examples of interest:
non-elementary hyperbolic and relatively hyperbolic groups~\cite{DGO}, the mapping class group
of all but finitely many surfaces of finite type~\cite[Theorem 2.19]{DGO}, 
${\rm Out}(F_n)$ for $n\geq 2$~\cite[Theorem 2.20]{DGO},
groups acting geometrically on a proper CAT(0)
space with a rank one isometry (\cite{Si-contr} and \cite[Theorem 2.22]{DGO}),
fundamental groups of several graphs of groups~\cite{MO}, small cancellation groups~\cite{GS-smallcanc}, finitely presented residually finite groups with positive first $\ell^2$-Betti number as well as groups of deficiency at least 2~\cite{Osinlast}, and many more.
In particular, Theorem~\ref{generalcase} widely generalizes Soma's previously mentioned results.

In order to prove Theorem~\ref{generalcase} we proceed as follows. 
We introduce a new seminorm $\|\cdot \|_{\bah}$ on exact bounded cohomology, which satisfies the inequality
$\|\cdot\|_\bah\geq \|\cdot \|_\infty$: for every finite subset $S$ of $\G$ and class $\alpha\in EH^{n}_b(\G)$ we set
$$
\|\alpha\|_S=\inf \{\|\delta\varphi\|_\infty\, |\, \varphi\in C^{n-1}(\G)^\G,\, [\delta\varphi]=\alpha,\, \varphi|_{S^{n}}=0\}\ ,
$$
and we define
$$
\|\alpha\|_\bah=\sup \{\|\alpha\|_S,\, S\subseteq \G,\, S \ \textrm{finite}\}\ \in \ [0,+\infty]\ .
$$

We denote by $N^n_0(\G)$ the subspace of elements  $\alpha\in EH^n_b(\G)$ such that $\|\alpha\|_\bah=0$, so that $N^n_0(\G)\subseteq N^n(\G)$
for every $n\in\mathbb{N}$. 
The key step in our proof of Theorem~\ref{generalcase} is then provided by the following:

\begin{thm_intro}\label{main:thm}
 The dimension of $N^3_0(F_2)$ 
has the cardinality of the continuum.
\end{thm_intro}

This already implies Theorem~\ref{generalcase} for non-abelian free groups (and, therefore, for all groups that admit an epimorphism on $F_2$, e.g.~for surface groups).
We then exploit results from~\cite{FPS} to reduce the general case to the case of free groups. In fact, 
an acylindrically hyperbolic group $\G$ contains a hyperbolically embedded subgroup $H$ which is virtually free-non-abelian~\cite{DGO,Os-acyl} (in fact, random subgroups satisfy this property~\cite{MS-random_hyp_emb}). Moreover, \cite[Corollary 1.2]{FPS}
implies that the inclusion $H\hookrightarrow \G$ induces  a surjection of $EH^3_b(\G)$ onto $EH^3_b(H)$, which we know to be infinite-dimensional from Theorem~\ref{main:thm}. 
This does not quite suffice to conclude, since we do not know whether the surjection $EH^3_b(\G)\to EH^3_b(H)$ does restrict to a surjection $N^3(\G)\to N^3(H)$.
This last fact would be true provided that the map $EH^3_b(\G)\to EH^3_b(H)$ is undistorted, according to the following:

\begin{defn_intro}\label{undistorted:defn}
A map $f\colon V\to W$ between seminormed vector spaces is \emph{undistorted} if there exists $k\geq 0$ such that
for every $\alpha\in f(V)$ there exists $\beta\in V$ with $f(\beta)=\alpha$ and $\|\beta\|\leq k\cdot \|\alpha\|$.  
\end{defn_intro}

Unfortunately, we are not able to show that the surjection $EH^3_b(\G)\to EH^3_b(H)$ is undistorted with respect to Gromov seminorms. 
In fact, Remark~\ref{distortion:qc} says that this cannot be achieved at the level of quasi-cocycles, and 
therefore undistortion is a rather delicate matter related to coboundaries, which makes it far from clear that this should even be true. Nevertheless, 
 in Section~\ref{sec:2} we prove the following:

\begin{thm_intro}\label{FPS:thm}
 Let $H$ be hyperbolically embedded in $\G$, let $n\geq 2$ and suppose that $H^{ n-1}(H)$ is finite-dimensional. 
 If we endow both $EH^n_b(\G)$  and $EH^n_b(H)$ with the seminorm $\|\cdot\|_\bah$, then
the inclusion $H\hookrightarrow \G$ induces an undistorted surjection
 $$
 EH^n_b(\G)\longrightarrow EH^n_b(H)\ .
 $$
\end{thm_intro}

This immediately implies that $\dim N^3(\G)\geq \dim N_0^3(\G)\geq \dim N^3_0(H)$,
thus allowing us to deduce Theorem~\ref{generalcase} from Theorem~\ref{main:thm}. Indeed, much more is true: 
due to the definition of $\|\cdot\|_\bah$, the fact that $\dim N_0^3(\G)$ is infinite-dimensional implies  
that 
there are many non-trivial  classes in $EH^3_b(\G)$ with vanishing seminorm, each of which can be represented by cocycles that vanish on arbitrarily big
subsets of $\G$. This quite counterintuitive phenomenon vividly illustrates the failing of excision for bounded cohomology.

%
%

\subsection*{Quasi-cocycles}
Bounded cohomology is often computed via suitable resolutions, that allow to better exploit the geometry of the group under consideration. 
For example, suppose that $\G$ acts on a set $X$.
Then we denote by $C^n(\G\curvearrowright X)$ (resp.~$C_b^n(\G\curvearrowright X)$) the space of maps (resp.~bounded maps)
from $X^{n+1}$ to $\R$, endowed with the $\G$-action defined by
$$
g\cdot \varphi(x_0,\ldots,x_n)=\varphi(g^{-1}x_0,\ldots,g^{-1}x_n)\ .
$$
The
obvious differential $\delta\colon C^n(\G\curvearrowright X)\to C^{n+1}(\G\curvearrowright X)$ preserves both $\G$-invariance and boundedness
of cochains, so one can define the 
bounded cohomology $H^*_b(\G\actson X)$ as the cohomology of the complex $C^*_b(\G\curvearrowright X)^\G$ of invariant bounded cochains.

The $\ell^\infty$-norm $\|\cdot\|_\infty$ on $C^n_b(\G\curvearrowright X)$ induces an $\ell^\infty$-seminorm
 on 
$H_b^n(\G\curvearrowright X)$, which is still denoted by $\|\cdot\|_\infty$.
Moreover,  if the action of $\G$ on $X$ is free, then 
$H^n_b(\G\curvearrowright X)$ is canonically isometrically isomorphic to $H^n_b(\G)$ for every $n\in\mathbb{N}$ (see Lemma~\ref{isometric:isom}).
In particular, $N^3(\G)$ is canonically isomorphic to the subspace
of elements of $H^3_b(\G\actson X)$ with vanishing seminorms. Every element with vanishing seminorm is exact (see Lemma~\ref{lem:Nexact}), so it can be represented
by a quasi-cocycle.
We are thus lead to investigate the space
$$
\QC^n(\G\actson X)=\{\varphi\in C^n(\G\actson X)\, |\, \|\delta\varphi\|_\infty<\infty \}
$$
of quasi-cocycles defined on $X$: namely, in order to prove that $H^3_b(\G)$  contains many elements with vanishing seminorm, we
will construct an uncountable family of invariant $2$-quasi-cocycles whose differential defines linearly independent bounded cohomology classes.

A crucial notion that keeps track of the seminorm of classes induced by quasi-cocycles
is the \emph{defect}:
just as in the case of quasi-morphisms, the \emph{defect} of a quasi-cocycle $\varphi\in \QC^n(\G\actson X)$ is given by
$$
\Def(\varphi)=\|\delta\varphi\|_\infty\ .
$$

\subsection*{A combinatorial volume form on hyperbolic $3$-manifolds fibering over the circle}
Let us now look more closely at the case we are interested in. Let $\G_0=F_2$ be the free group generated by the elements $a,b$,
and let us identify $\G_0$ with the fundamental group of the punctured torus $\Sigma$, in such a way that the conjugacy class of the commutator $[a,b]=a^{-1}b^{-1}ab$ corresponds to the isotopy class
of a simple closed curve winding around the puncture. 
We fix a group automorphism $\psi\colon \G_0\to \G_0$ induced by a pseudo-Anosov homeomorphism
$f\colon \Sigma\to \Sigma$. The automorphism $\psi$ preserves the conjugacy class of the commutator $[a,b]$, so, up to conjugacy,
we may suppose that $\psi([a,b])=[a,b]$.
The mapping torus
$$
M=\Sigma\times [0,1] \big/_\sim\, ,\qquad (x,0)\sim (f(x),1)
$$
has fundamental group isomorphic to the semidirect product
$\G = \G_0 \rtimes_{\psi} \Z$, where the generator $t$ of $\Z$ acts on $\G_0$ as follows:
 $ t g t^{-1}= \psi( g) $ for every $g\in \G_0$. The (cusp) subgroup $H$ of $\G$ is the subgroup generated by
 $t$ and (the image of) $[a,b]\in\G_0$, and it is isomorphic to $\Z\oplus\Z$.
 
 Recall that the pair $(\G,H)$ is relatively hyperbolic, either by Thurston's hyperbolization for manifolds fibering over the circle~\cite{Otal} and a fundamental result by Farb~\cite{Fa-relhyp}, or just by a combination theorem for relative hyperbolicity~\cite[Theorem 4.9]{MjReeves-combination}.


Starting from a Cayley graph of $\G$, one can construct a \emph{cusped graph} $X$ by gluing a copy of a combinatorial horoball based on $H$ to each left coset of $H$ in $\G$; we outline the construction in Section~\ref{bicombing:sec}.
It was first described by Groves and Manning in~\cite{GM}, and a similar construction is described in~\cite{Bow-relhyp}. 



The group $\G$  acts freely on $X$ by isometries,
therefore the bounded cohomology of $\G_0$ can be isometrically computed by the complex 
$C^*_b({\Gamma_0 \actson X})^{\G_0}$. Moreover, being obtained by adding horoballs to (the Cayley graph of) a relatively hyperbolic group,
the graph $X$ is Gromov hyperbolic, and supports a quasi-geodesic homological bicombing with useful filling properties (see Section~\ref{bicombing:sec}). Indeed, $X$ is quasi-isometric to the hyperbolic $3$-space, and the bicombing may be exploited
to construct a combinatorial version of the hyperbolic volume form. In fact, since the cochains arising in our argument must all be $\G_0$-invariant, the combinatorial cocycles we construct should be thought of as
volume forms on the differential counterpart of $X/\G_0$, that is the infinite cyclic covering $M_0$ of $M$ associated to $\G_0<\G=\pi_1(M)$. 

As it is customary when dealing with ``quasifications'' of algebraic or differential notions, the direct construction of a volume cocycle on $X$ runs into difficulties, due to the fact that the coarse version of a cocycle needs not be a
cocycle. Therefore, in Section~\ref{sec:4} we rather construct a $\G_0$-invariant primitive of a volume form. Such primitive 
is a quasi-cocycle, and  its differential (which is automatically closed)
provides a combinatorial version
of the volume form on $M_0\cong \Sigma\times \R$. Following Soma's strategy, in order to construct an infinite-dimensional subspace of $EH^3(\G_0)$ out of this primitive,
we just consider the suitably chosen collection of quasi-cocycles obtained by taking the product of the original primitive with a collection of real functions on $M_0\cong \Sigma\times\R$. These functions are themselves constructed
by composing the projection $\Sigma\times\R\to\R$ with Lipschitz maps of $\R$ into itself. The outcome of this procedure is summarized by the following result, which provides the key ingredient for the proof
of Theorem~\ref{main:thm}:

\begin{thm_intro}\label{volumeform:thm}
 Let $\mathcal{L}(\mathbb{Z},\mathbb{R})$ be the space of Lipschitz real functions on $\mathbb{Z}$. There exist a constant $C>0$ and a linear map
 $$
 \alpha\colon \mathcal{L}(\mathbb{Z},\mathbb{R})\to \QCa^2({\G_0\actson X})^{\G_0}\,
 $$
 such that the following conditions hold:
\begin{enumerate}
 \item $\|\delta\alpha(f))\|_\infty \leq C\cdot \lip(f)$ for every $f\in\mathcal{L}(\mathbb{Z},\mathbb{R})$;
 \item $[\delta\alpha(f)]=0$ in $H^3_b(\G_0\actson X)\cong H^3_b(\G_0)$ if and only if $f$ is bounded.
\end{enumerate}
\end{thm_intro}


\subsection*{Volumes of mapping tori}
We believe that the techniques developed in this paper, and especially the combinatorial description of a volume form, will have application in other contexts as well. As a first example in this direction, in the appendix 
we give a cohomological proof
of a volume estimate for hyperbolic 3--manifolds fibering over the circle, under a coboundedness assumption. Recall that a pseudo-Anosov homeomorphism $\psi\colon \Sigma_g\to\Sigma_g$ is $\epsilon$-cobounded if, 
denoting by $l$ its axis in the Teichm\"uller space endowed with the Teichm\"uller metric, the projection of $l$ is contained in the $\epsilon$-thick part $\cal M_g^\epsilon$ of the moduli space.
We denote by $\tau(\psi)$ the translation length of $\psi$ on the Teichm\"uller space endowed with the Teichm\"uller metric.

\begin{thm_intro}\label{thm:Aintro}
There exists a constant $C>0$ depending only on $\epsilon$ and $g$ such that, for any $\epsilon$-cobounded pseudo-Anosov $\psi:\Sigma_g\to\Sigma_g$, we have
$$\vol(M_\psi)\geq C\tau(\psi).$$
\end{thm_intro}
This result was originally proven by Brock with completely different techniques \cite{Brock}.  In fact, we emphasize that our proof actually gives an estimate on the simplicial volume of $M_\psi$, and we then deduce the volume estimate from the well-known proportionality between volume and simplicial volume for hyperbolic manifolds. However, in no other part of the proof we use the fact that $M_\psi$ is hyperbolic.

We decided to include such a result only in an appendix because the setting is slightly different from the rest of the paper. Since we only deal with compact manifolds, many of the technicalities involved in the main paper are not needed for this application. For this reason, a reader interested only in the construction of a combinatorial cocycle representing the volume form might want to read the appendix first.

\subsection*{Tl;dr: the definition of the quasi-cocycles} For future reference and to help the reader find the relevant definitions, we list here all notions involved in the construction of our quasi-cocycles, and we give the definition of the quasi-cocycles themselves.

\begin{itemize}
 \item $\Gamma_0$ is the free group on two generators, $a,b$.
 \item $\psi:\Gamma_0\to \Gamma_0$ is an automorphism induced by a pseudo-Anosov, and it preserves the commutator $[a,b]$.
 \item $\Gamma$ is the semidirect product $\Gamma_0\rtimes_\psi \mathbb Z$, and $H<\Gamma$ is the subgroup generated by $[a,b]$ and the stable letter $t$.
 \item $X$ is the cusped graph of $(\Gamma,H)$ (Definition \ref{defn:cusped}), which is $\delta$-hyperbolic. Vertices of $X$ are pairs $(g,n)$ with $g\in G$, $n\in\mathbb N$.
 \item $p:X^{(0)}\to \Gamma_0$ is defined by $\psi(g_0t^k,n)=g_0$, where $g_0\in\Gamma_0$. Also, $\theta:X^{(0)}\to \mathbb Z$ is defined by $\theta(g_0t^k,n)=k$, where $g_0\in\Gamma_0$.
 \item $\rho:\Gamma_0\to\isom^+(\mathbb{H}^2)$ is a hyperbolization, and $[a,b]$ fixes $\overline{q}\in\partial \mathbb H^2$. For $x_0,x_1,x_2$ vertices of $X$, the sign $\epsilon(x_0,x_1,x_2)$ is $1$, $-1$ or $0$ depending on the orientation of the ideal triangle of $\mathbb H^2$ with vertices $\rho(p(x_i))\overline q$ (Subsection \ref{subsec:area}).
 \item $\mathcal X$ is the Rips complex on $X$ with constant $\kappa\geq 4\delta+6$ (Definition \ref{defn:rips}).
 \item $\phi$ is a relative filling map, i.e. a map from $X^3$ to $2$-cochains of $\mathcal X$ (Proposition \ref{ofeuihwln}).
 
 \par\smallskip
 
 Let now $f:\mathbb Z\to\R$ be Lipschitz.
 
 \par\smallskip
 
 \item The simplicial $2$--cochain $F_f$ on $\calX$ (Definition \ref{definition F_f}) is the one such that, if $\sigma$ is a $2$-simplex in $\calX$ with vertices $(x_0,x_1,x_2)\in X^3$, then
$$
F_f(\sigma) = \varepsilon(x_0,x_1,x_2) \frac{\sum_{i=0}^2 f(\theta(x_i))}{3}\ .  $$

\item Finally, the quasi-cocycle $\alpha_f:X^3\to \mathbb R$ is defined by
$$\alpha_f(x_0,x_1,x_2)=F_f(\phi(x_0,x_1,x_2)).$$ 
\end{itemize}

Proposition \ref{quasi-cocycle:prop} says that $\alpha_f$ is indeed a quasi-cocycle, and that its defect is bounded by a universal constant times the Lipschitz constant of $f$. Proposition \ref{prop:final} says that the coboundary $\delta \alpha_f$ is trivial in bounded cohomology if and only if $f$ is bounded.
\subsection*{Open questions and directions for further research}
Is quasification indeed essential in order to prove Theorem~\ref{main:thm}? Surprisingly enough, 
it seems that studying
genuine differential forms on hyperbolic manifolds is much harder than working with quasi-cocycles
on discrete models for $F_2$. For example, if $M_0\cong\Sigma\times \R$ is the hyperbolic manifold introduced above, where $\Sigma$ is a punctured torus, integration
over straight simplices induces a map from the space of pointwise bounded differential $3$-forms on $M_0$ to bounded group cochains of degree 3.  Understanding the kernel of this map
is unexpectedly difficult, and this implies that it is not trivial to detect when distinct differential forms represent the same bounded class, i.e.~how much freedom one can enjoy in varying the 
differential representatives of a fixed bounded class. 
We refer the reader to~\cite{BIgeo, Wienhard} for a discussion of this topic.  In~\cite{KimKim} Kim and Kim proved, for example, that if $M$ is a complete, connected, oriented, locally symmetric space of infinite volume, 
then the Cheeger isoperimetric constant of $M$ is positive if and only if the Riemannian volume form on $M$ admits a bounded primitive. They also showed that if $M$ is a complete, connected, oriented, 
$\R$-rank one locally symmetric space of infinite volume with dimension at least 3, then the volume form of $M$ defines a non-trivial bounded cohomology class if and only if the Cheeger constant of $M$ vanishes. 
We pose here the following:

\begin{quest_intro}
Let $n\geq 3$ and
 let $M$ be a hyperbolic $n$-manifold of infinite volume with vanishing Cheeger constant. Is it possible to characterize the space of $n$-forms on $M$ admitting a bounded primitive? For example, is it true that a compactly supported
$n$-form on $M$ admits a bounded primitive?
\end{quest_intro}

This question is tacitly faced in~\cite{Soma1} in the case when $M$ is the cyclic covering of a $3$-manifold fibering over the circle with fiber a \emph{closed} surface. Soma's analysis involves a careful study of the relationship
between the hyperbolic and the singular Sol volume forms supported by such a manifold. Adapting his arguments to the case when the fiber is a punctured surface seems very delicate.

\bigskip

Monod and Shalom showed the importance of bounded cohomology with
coefficients in $\ell^2(\G)$ in the study of rigidity of $\G$~\cite{MonShal0,MonShal}, and
proposed the condition $H^2_b(\G,\ell^2(\G))\neq 0$ 
as a cohomological definition of
negative curvature for groups. More in general, bounded cohomology with
coefficients in $\ell^p(\G)$, $1\leq p<\infty$
has been widely studied as a powerful
tool to prove (super)rigidity results (see e.g.~\cite{Hamen,CFI}). It is still unknown whether $H^3_b(F_2,\ell^2(F_2))$ vanishes or not. We hope that our combinatorial approach to the construction of non-trivial classes
(with trivial real coefficients) could be of use in the context of more general coefficient modules.

\subsection*{Plan of the paper}
In Section \ref{background:sec} we recall some basic facts on bounded cohomology, and introduce the various (co)homological complexes we will need in the paper. 
In Section \ref{sec:2} we introduce the seminorm $\|\cdot\|_\bah$ and prove Theorem \ref{FPS:thm} building on results from \cite{FPS}. We also show how Theorem~\ref{generalcase} may be reduced to Theorem~\ref{main:thm}.
Following \cite{GM} and \cite{Franceschini2},
in Section \ref{bicombing:sec} we describe a combinatorial bicombing with good filling properties on a suitably chosen Rips complex  associated to a relatively hyperbolic pair.
In Section \ref{sec:4} we construct a family of $3$-dimensional combinatorial volume forms on the free group on two generators, and we prove Theorems \ref{volumeform:thm} and~\ref{main:thm}.
Finally, in the appendix, we discuss applications of our techniques to obtain bounds on the volume of compact hyperbolic manifolds and prove Theorem \ref{thm:Aintro}.
\subsection*{Acknowledgements}
We thank Mladen Bestvina for an important suggestion that made it possible to construct a combinatorial version of Soma's classes, and Ken Bromberg for very useful discussions. Part of this work was carried out during the conference "Ventotene 2015 -- Manifolds and Groups". Beatrice Pozzetti was partially supported by the SNF grant P2EZP2\_159117.
This material is based upon work supported by the National Science
Foundation under Grant No. DMS-1440140 while the fourth author was in
residence at the Mathematical Sciences Research Institute in Berkeley,
California, during the Fall 2016 semester.
The third and fourth author would like to thank the Isaac Newton Institute for Mathematical Sciences, Cambridge, for support and hospitality during the programme ``Non-Positive Curvature Group Actions and Cohomology'' where work on this paper was undertaken. This work was supported by EPSRC grant no EP/K032208/1.

\section{Preliminaries on bounded cohomology}\label{background:sec}

Let $\G$ be a group and let $X$ be a set on which $\G$ acts on the left. 
We set
$$
C^n(\G\actson X)=\{\varphi\colon X^{n+1}\to\R\}\ ,
$$
and we endow $C^n(\G\actson X)$ with the left $\G$-action defined by
$$
g\cdot\varphi(x_0,\ldots,x_n)=\varphi(g^{-1}x_0,\ldots,g^{-1}x_n)\ .
$$
For every $n\in\mathbb{N}$ we also define the differential 
$\delta\colon\ C^n(\G\actson X)\to C^{n+1}(\G\actson X)$ by setting
$$
\delta\varphi (x_0,\ldots,x_{n+1})=\sum_{i=0}^{n+1} (-1)^i \varphi(x_0,\ldots,\widehat{x}_i,\ldots,x_{n+1})\ ,
$$
and we put on $C^n(\G\actson X)$ the norm defined by
$$
\|\varphi\|_\infty=\sup \{|\varphi(x_0,\ldots,x_n)|\, ,\ (x_0,\ldots,x_{n})\in X^{n+1}\} \in \ [0,+\infty]\ .
$$
We denote by $C^n_b(\G\actson X)\subseteq C^n(\G\actson X)$ the subspace of bounded cochains, and we observe that 
$\|\cdot\|_\infty$ restricts to a finite norm on $C^n_b(\G\actson X)$. 

If $V$ is a vector space endowed with a linear $\G$-action, we denote by $V^\G\subset V$ the subspace of elements that are fixed by every element of $\G$. 
The differential defined above commutes with the action of $\G$ and sends
bounded cochains to bounded cochains. Therefore, we can consider the cohomology of the complexes
$C^n(\G\actson X)^\G$ and $C^n_b(\G\actson X)^\G$, which we denote respectively by $H^n(\G\actson X)$ and $H^n_b(\G\actson X)$.   
If $X=\G$, endowed with the left action by translations, one gets back the usual (bounded) cohomology $H^*_{(b)} (\G)$
of $\G$.




For every basepoint $x\in X$, we consider the $\G$-equivariant chain map
$$
w^*_{x}\colon C^*_{b}(\G\actson X)\to C^*_b(\G)\ ,
$$
$$
w^n_{x}(\varphi)(g_0,\ldots,g_n)=\varphi(g_0x,\ldots,g_nx)\ .
$$
 With a slight abuse, we denote by $w^n_x$ also the induced map $w^n_x\colon H^n_b(\G\actson X)\to H^n_b(\G)$
 on bounded cohomology.

 \begin{lemma}\label{isometric:isom}
Suppose that the action of $\G$ on $X$ is free. Then for every $x\in  X$ the map $w^n_x:H^n_b(\G\actson X)\to H^n_b(\G)$ is a natural isometric isomorphism.
\end{lemma}
\begin{proof}
Free actions are very special instances of amenable actions, so the conclusion follows e.g.~from \cite[Theorem 7.5.3]{Monod}.
\end{proof}

\subsection{The predual chain complex}
In order to show that the cocycles we are going to construct are non-trivial, we will need to evaluate them on appropriate chains. 
Let us fix an action of a group $\G$ on a set $X$ as in the previous section. 
For every $n\geq 0$ we denote by $C_n(X)$ the real vector space 
with basis $X^{n+1}$. Elements of $X^{n+1}$ will be often called  $n$-simplices, 
since they are the $n$-simplices of the full simplicial complex with vertices in $X$.
As usual, we say that
 an $n$-simplex  is supported on a subset $S\subseteq X$ if
all its vertices lie in $S$, and the subspace of $C_n(X)$
generated  by simplices supported on $S$ is denoted by $C_n(S)$.
We also endow $C_n(X)$ with the $\ell^1$-norm defined by
$$
\left\|\, \sum_{\overline{x}\in X^{n+1}} a_{\overline{x}} \overline{x}\, \right\|_1=\sum_{\overline{x}\in X^{n+1}} |a_{\overline{x}}|\ .
$$
If $\overline{x}=(x_0,\ldots,x_n)\in C_n(X)$, we denote by $\partial_j \overline{x}=(x_0,\ldots,\widehat{x_j},\ldots,x_n)\in
C_{n-1}(X)$ the $j$-th face of $\overline{x}$, and we set $\partial \overline{x}=\sum_{j=0}^n (-1)^j \partial_j\overline{x}$.
Observe that it readily follows from the definitions that the diagonal $\G$-action on $X^{n+1}$ induces an isometric $\G$-action on $C_n(X)$. 
We denote by $C_n(\G\actson X)$ the normed space $C_n(X)$ equipped with this action. 

The dual notion to cochain invariance is chain coinvariance.  
We define the space of \emph{coinvariants} of $C_n(\G\actson X)$
as the quotient space
$$
C_n(\G\actson X)_{\G}=C_n(\G\actson X)\big/ W
$$ 
where $W$ is the subspace of 
$C_n(\G\actson X)$
spanned by the elements of the form $g\cdot c - c$, as $c$ varies in 
$C_n(\G\actson X)$ and $g$ varies in $\G$.
We endow $C_n(\G\actson X)_{\G}$
with the quotient seminorm (which  is a norm).
Since the $\G$-action on 
$C_n(\G\actson X)$ commutes with the boundary operator, 
$C_*(\G\actson X)_\G$ is naturally a chain complex, whose homology will be denoted by
$$
H_*(\G\actson X)\ .
$$
The $\ell^1$-norm on 
$C_n({\G\actson X})_\G$
induces a seminorm on $H_*({\G\actson X})$, which will still be denoted by $\|\cdot\|_1$.

Since invariant cochains vanish on the subspace $W$ previously defined, evaluation of cochains on chains induces a pairing
$$
\langle\cdot,\cdot\rangle \colon C^n({\G\actson X})^\G\times C_n({\G\actson X})_\G\to\R\ ,
$$
which in turn induces pairings
\begin{align*}
\langle\cdot,\cdot\rangle &\colon H^n({\G\actson X})\times H_n({\G\actson X})\to\R\ ,\\
\langle\cdot,\cdot\rangle &\colon H_b^n({\G\actson X})\times H_n({\G\actson X})\to\R\ .
\end{align*}
It readily follows from the definitions that
$$
\langle \alpha,\beta\rangle \leq \|\alpha\|_\infty \cdot \|\beta\|_1
$$
for every $\alpha\in H^n_b(\G\actson X)$, $\beta\in H_n(\G\actson X)$.
As a first application of the pairing between homology and cohomology, we show that bounded coclasses with vanishing seminorm are exact:
\begin{lemma}\label{lem:Nexact}
We have $N^n(\G)\subseteq EH^n_b(\G)$.
\end{lemma}
\begin{proof}
By the Universal Coefficient Theorem, the pairing $\langle\cdot,\cdot\rangle \colon H^n(\G)\times H_n(\G)\to\R$ induces an isomorphism
between $H^n(\G)$ and the dual of $H_n(\G)$. Therefore, in order to conclude it is sufficient to observe that, if $c^n\colon H^n_b(\G)\to H^n(\G)$ is the comparison map, then
$$
|\langle c^n(\alpha),\beta\rangle|=
|\langle \alpha,\beta\rangle|\leq \|\alpha\|_\infty \|\beta\|_1= 0
$$
for every $\alpha\in N^n(\G)$, $\beta\in H_n(\G)$. 
\end{proof}

\subsection{Degenerate chains and alternating cochains}
In later computations it will be convenient to  neglect degenerate simplices (i.e.~simplices with non-pairwise distinct vertices).  
To this aim,
let us denote by $\mathfrak{S}_{n+1}$ the group of permutations
of the set $\{0,\ldots,n\}$, and by ${\rm sgn}(\sigma)=\pm 1$
the sign of $\sigma$, 
for every $\sigma\in\mathfrak{S}_{n+1}$.

Then we may define an alternating linear operator
$
\alt_n\colon C_n(\G\actson X)\to C_n(\G\actson X)
$
by setting, for every $\overline{x}=(x_0,\ldots,x_n)\in X^{n+1}$,
$$
\alt_n(\overline{x})=\frac{1}{(n+1)!} \sum_{\sigma\in\mathfrak{S}_{n+1}} {\rm sgn}(\sigma)
(x_{\sigma(0)},\ldots,x_{\sigma(n)})\ .
$$
We say that a chain $c\in C_n(\G\actson X)$ is \emph{degenerate} if $\alt_n(c)=0$, and we denote by $D_n(\G\actson X)$ the subspace of degenerate chains. 
We observe that $D_n(\G\actson X)$ contains (strictly, unless $X$ is a point) the
space spanned by degenerate simplices.

It is immediate to check that $\alt_*$ 
commutes with the boundary operator and with the action of $\G$. Therefore,
it descends to a chain map $\alt_*\colon C_*({\G\actson X})_\G\to C_*({\G\actson X})_\G$, which will still be denoted by $\alt_*$. If $D_n(\G\actson X)_\G$ denotes the image
of $D_n(\G\actson X)$ in $C_n({\G\actson X})_\G$, we then
define
\emph{reduced}
chains by setting
\begin{align*}
C_*({\G\actson X})_{{\red}} & = C_*({\G\actson X})/D_*(\G\actson X), \\
C_*({\G\actson X})_{{\red},\G} & = C_*({\G\actson X})_{\G}/D_*(\G\actson X)_\G\ .
\end{align*}

It is well known that the homology of the complex $C_*({\G\actson X})_{{\red},\G}$,
endowed with the obvious quotient seminorm, is isometrically isomorphic to $H_n({\G\actson X})$: indeed, this easily follows from the fact that
alternation is homotopic to the identity (on any complex where it is defined), and norm non-increasing (see e.g.~\cite[Appendix B]{Fuji_Man}).

Dually, one may define alternating cochains by setting, for every $\varphi\in C^n(\G\actson X)$, 
$$
\alt^n(\varphi)(\overline{x})=\varphi(\alt_n(\overline{x}))
$$
for every $\overline{x}\in X^{n+1}$. The map $\alt^n$ 
commutes with the differential and with the action of $\G$, thus defining
a norm non-increasing chain self-map of the complex $C^*({\G\actson X})$. We denote by
$$
C^*_{\alt} ({\G\actson X})=\alt^*(C^* ({\G\actson X}))
$$
the space of alternating cochains, and we set
$$
C^*_{b,\alt} ({\G\actson X})=C^*_{\alt} ({\G\actson X})\cap C^*_b ({\G\actson X})\ .
$$

Again, the inclusion of alternating cochains into generic cochains induces an isometric isomorphism between the cohomology of the complex 
$C^*_{b,\alt} ({\G\actson X})$ and $H^*_b({\G\actson X})$. Moreover, 
since alternating cochains vanish on degenerate chains, 
there is a well-defined pairing
$$
\langle\cdot,\cdot\rangle \colon C^n_{b,\alt}({\G\actson X})^\G\times C_n({\G\actson X})_{{\red},\G} \to\R\ ,
$$
which, under the identifications previously mentioned, induces the pairing between
$H^n_{b}({\G\actson X})$ and $H_n({\G\actson X})$ introduced above. 

We will denote by $\QCa^n({\G\actson X})$ the space of alternating quasi-cocycles on $X$, i.e.~ the set
of alternating cochains with bounded differential.

\subsection{Simplicial (co)chains}
In this paper we will study the cochain modules $C^*(\G\actson X)$, $C^*_b(\G\actson X)$ in the case when $X$ is the set of vertices of a suitably augmented
Cayley graph of  $\G$ (see Section~\ref{bicombing:sec}). 
A key step in our arguments will be based on 
the fact that $\G$ is relatively hyperbolic, which implies that $\G$ satisfies (relative) isoperimetric inequalities in every degree.
In order to deal with higher dimensional fillings, it will be convenient, rather than considering cochains in $C^n(\G\actson X)$ (or $C^n_b(\G\actson X)$), to consider simplicial cochains
on suitably defined simplicial complexes related to $X$ (like the augmented Cayley graph having $X$ as set of vertices, or some Rips complex over $X$).

For every simplicial complex $Y$, we denote by $(C_*^\Delta(Y),\partial)$ the chain complex of real simplicial chains on $Y$, endowed with the $\ell^1$-norm $\|\cdot \|_1$ such that
$$
\left\|\sum_{i\in I} a_i\sigma_i\right\|_1=\sum_{i\in I} |a_i|
$$
for every reduced sum $\sum_{i\in I} a_i\sigma_i\in C_*^\Delta(Y)$. 
The module $C_n^\Delta(Y)$ is the real vector space with basis $$\{(y_0,\ldots,y_n)\, |\, \{y_0,\ldots,y_n\}\ {\rm is\ a\ simplex\ of}\ Y\}$$
(in particular, every $n$-simplex of $Y$ gives rise to $(n+1)!$ simplices in $C_n^\Delta(Y)$, and to many other degenerate ones
in degree bigger than $n$). As it is customary in the literature, we denote by $[y_0,\ldots,y_n]$ (rather than by $(y_0,\ldots,y_n)$) the elements of the canonical
basis of $C^\Delta_n(Y)$.
If $[y_0,\ldots,y_n]$ is any such element, then we set $\supp([y_0,\ldots,y_n])=\{y_0,\ldots,y_n\}\subseteq  Y^{(0)}$, 
and if $c=\sum_{i\in I} a_i\sigma_i$ is a chain in reduced form, then we set
$\supp(c)=\bigcup_{i\in I}\supp(\sigma_i)$. 

Just as above, we define 
a  chain map 
$
\alt_n\colon C^\Delta_n(Y)\to C_n^\Delta(Y)
$
by setting
$$
\alt_n([y_0,\ldots,y_n])=\frac{1}{(n+1)!} \sum_{\sigma\in\mathfrak{S}_{n+1}} {\rm sgn}(\sigma)
([y_{\sigma(0)},\ldots,y_{\sigma(n)}])
$$
for every $[y_0,\ldots,y_n]\in C_n^\Delta(Y)$.

A chain $c\in C_n(Y)$ is \emph{degenerate} if $\alt_n(c)=0$, and one may define the complex $ C_*^\Delta(Y)_{{\red}} $ of reduced simplicial chains as the quotient
of $C_*(Y)$ by the subspace of degenerate chains. We will simply denote by $[y_0,\ldots,y_n]$ (and call it a ``simplex'') also the class of 
$[y_0,\ldots,y_n]$ in $C_n^{\Delta}(Y)_\red$, so that, for example, we will be allowed to write that $[y_0,y_1]=-[y_1,y_0]$ 
in $C_1^{\Delta}(Y)_\red$. If one fixes a total ordering $<$ on the set of vertices of $Y$, then a basis of 
$C_n^\Delta(Y)_\red$ is given by the classes of the non-degenerate elements $[y_0,\ldots,y_n]\in C_n^\Delta(Y)$ such that $y_0<\ldots<y_n$.
We say that a simplex $\{y_0,\ldots,y_n\}$ \emph{appears} in  
a reduced chain $\overline{c}\in C_n^{\Delta}(Y)_\red$ 
if, when writing $\overline c$ as a linear combination of the elements of the above basis,
the coefficient of the  unique element corresponding to $\{y_0,\ldots,y_n\}$ is not null.
We then define the support $\supp(\overline{c})$ of $\overline{c}$ as the union of the sets of vertices 
of all the simplices appearing in $\overline{c}$. Equivalently, $\supp(\overline{c})$ is the smallest possible support
of any chain $c\in C_n^\Delta(Y)$ projecting to $\overline{c}$.

Also observe that the $\ell_1$-norm on $C_n^\Delta(Y)$ induces an $\ell^1$-norm on $C_n^{\Delta}(Y)_\red$, that will still be denoted by $\|\cdot \|_1$.

If $\G$ acts on $Y$ via simplicial automorphisms, then we denote by $C_*^\Delta(Y)_\G$ the complex of coinvariants of $C_*^\Delta(Y)$. Just as before,
$C_*^\Delta(Y)_\G$ is the quotient of $C_*^\Delta(Y)$ by the submodule generated by the chains of the form $(c-g\cdot c)$, $c\in C_*^\Delta(Y)$, $g\in\G$. 
The chain map $\alt_*$ 
commutes with  the action of $\G$, thus descending to a map $\alt_*\colon C_*^\Delta(Y)_\G\to C_*^\Delta(Y)_\G$, which will still be denoted by $\alt_*$. We will
denote by $C_*^{\Delta}(Y)_\red$ (resp.~$C_*^{\Delta}(Y)_{\red,\G}$) the complex of reduced (resp. reduced and coinvariant) cochains, i.e.~the quotient of $C_*^\Delta(Y)$ 
(resp.~of $C_*^{\Delta}(Y)_\G$)
by the kernel of the alternation map.

It is well known that, if $Y$ is contractible and $\G$ acts freely on $Y$, the homology of the complexes $C_*^\Delta(Y)_\G$, $C_*^{\Delta}(Y)_{\red,\G}$ is (not isometrically!) isomorphic to the homology of $\G$.
One may wonder whether also the computation of bounded cohomology could take place in the context of simplicial cochains. However, this is almost never the case:
for example, if $Y/\G$ is compact, then 
every invariant simplicial cochain on $Y$ is bounded, while there may well
exist cohomology classes in $H^n(\G)$ which do not admit any bounded representative. 

\section{Controlled extensions of quasi-cocycles}\label{sec:2}
This section is devoted to the description of some elementary properties of the norm 
$\|\cdot\|_\bah$ defined in the introduction, and to 
the proof of
Theorem \ref{FPS:thm}. We fix a group $\G$, and we work with the standard resolution $(C^n_b(\G),\delta)$ computing $H^n_b(\G)$.

\subsection{The seminorm $\|\cdot\|_\bah$}
 Recall from the introduction that,
for every  class $\alpha\in EH^{n+1}_b(\G)$, we have set
$$
\|\alpha\|_\bah=\sup \{\|\alpha\|_S,\, S\subseteq \G,\, S \ \textrm{finite}\}\ \in \ [0,+\infty]\ ,
$$
where
$$
\|\alpha\|_S=\inf \{\|\delta\varphi\|_\infty\, |\, \varphi\in C^n(\G)^\G,\, [\delta\varphi]=\alpha,\, \varphi|_S=0\}\ .
$$

In~\cite[Section 5.34]{Gromovbook}, Gromov  called \emph{functorial} any seminorm (on singular homology of topological spaces) 
with respect to which every continuous map induces a norm non-increasing morphism. 
The following result ensures that $\|\cdot\|_\bah$ satisfies the obvious analogous of functoriality for seminorms on bounded cohomology of groups:

\begin{lemma}\label{functorial}
 Let $\psi\colon \G\to \G'$ be a homomorphism. Then the induced map
 $$
 \psi^*\colon (EH^n_b(\G'),\|\cdot\|_\bah)\to  (EH^n_b(\G),\|\cdot\|_\bah)
 $$
 is norm non-increasing for every $n\in\mathbb{N}$.
\end{lemma}
\begin{proof}
 Take $\alpha\in EH^n_b(\G')$, and let $\varepsilon>0$ be given. Let also $S\subseteq \G$ be an arbitrary finite set. Of course the set
 $S'=\psi(S)$ is finite, so we can find an element $\varphi'\in C^{n-1}(\G')^{\G'}$ such that $[\delta \varphi']=\alpha$ in $EH^n_b(\G')$,
 $\varphi'|_{S'}=0$, and $\|\delta\varphi'\|_\infty\leq \|\alpha\|_\bah +\varepsilon$. Let now $\varphi=\psi^*\varphi'$. By construction we have
 $\varphi|_S=0$, $[\delta\varphi]=[\delta \psi^*\varphi']=[\psi^*\delta\varphi']=\psi^*(\alpha)$, and
 $\|\delta\varphi\|_\infty\leq \|\delta\varphi'\|_\infty\leq \|\alpha\|_\bah+\varepsilon$. Hence $\|\psi^*(\alpha)\|_S\leq \|\alpha\|_\bah+\varepsilon$.
 Due to arbitrariness of $S$ we then have $\|\psi^*(\alpha)\|_\bah\leq \|\alpha\|_\bah+\varepsilon$, whence the conclusion since $\varepsilon$ is arbitrary.
 \end{proof}

\begin{cor}\label{amenable:inj}
 Let $\psi\colon \G\to \G'$ be a surjective homomorphism with amenable kernel. Then 
 $
 \psi^*\colon H^n_b(\G') \to  H^n_b(\G)
 $
 induces an injection
 $$
 N_0^n(\G')\hookrightarrow N_0^n(\G)
 $$
for every $n\in\mathbb{N}$.
\end{cor}
\begin{proof}
 It is well known that an epimorphism with amenable kernel induces an isomorphism in bounded cohomology (see e.g.~\cite{Gromov,Ivanov}), so the conclusion follows from Lemma~\ref{functorial}.
\end{proof}

\begin{quest}
Let $\psi\colon \G\to \G'$ be a surjective homomorphism with amenable kernel.
 Then the isomorphism $\psi^*$ induced by $\psi$ on bounded cohomology is isometric with respect to Gromov's seminorm (see e.g.~\cite{Gromov,Ivanov}).
 Is it true that $\psi^*$ also preserves the seminorm $\|\cdot\|_\bah$ on exact bounded cohomology? Or could the seminorm $\|\cdot\|_\bah$ be used to distinguish the (exact) bounded cohomology
 of $\G$ from the (exact) bounded cohomology of $\G'$ (as seminormed spaces)?
\end{quest}

The norm $\|\cdot\|_\bah$ is only interesting in degrees strictly bigger than 2:

\begin{lemma}\label{degree2:infinite}
 For each non-zero $\alpha$ in $H^2_b(\G)$, $\|\alpha\|_\bah=\infty$.
\end{lemma}
\begin{proof}
Let $C_n^{\ell^1}(\G)_\G$ be the metric completion of $C_n(\G)_\G$ with respect to the $\ell^1$-norm. Being bounded, the differential $\partial_n\colon C_n(\G)_\G\to C_{n-1}(\G)_\G$ extends
to $\ell^1$-chains, thus defining a complex whose homology is denoted by $H_*^{\ell_1}(\G)$. The pairing between
$H^n_b(\G)$ and $H_n(\G)$ extends to a pairing between $H^n_b(\G)$ and $H_n^{\ell_1}(\G)$. By~\cite[Theorem 2.3 and Corollary 2.7]{Matsu-Mor}, 
this pairing induces an isomorphism between $H^2_b(\G)$ and the dual of $H_2^{\ell_1}(\G)$. 
Therefore, it is sufficient to show that if 
 $\alpha\in H^2_b(\G)$ is any element with $\|\alpha\|_\bah=M<\infty$, then $\alpha$ vanishes on every class in $H_2^{\ell_1}(\G)$. 
 So let $\beta\in C_2^{\ell_1}(\G)$ be an $\ell^1$-cycle, and let $\varepsilon$ be given. We can find a decomposition $\beta= \beta_1+\beta_2$ 
 such that $\|\beta_2\|_1<\varepsilon$ and $\beta_1$ is supported on a finite set $S\subseteq \G$. Since $\|\alpha\|_\bah\geq \|\alpha\|_S$, 
 we can find a representative $a$ of $\alpha$ vanishing on $S$ with Gromov norm smaller than $M+1$. This implies
 $$\langle\alpha,[\beta]\rangle=\langle a,\beta_1+\beta_2\rangle=\langle a,\beta_2\rangle\leq\varepsilon(M+1)\ .$$
 By the arbitrariness of $\varepsilon$, this implies that $\langle\alpha,[\beta]\rangle=0$, as desired.
\end{proof}

The rest of the section is devoted to the proof of Theorem~\ref{FPS:thm}.
We first 
describe  an easy characterization of the seminorm $\|\cdot\|_\bah$ defined in the introduction.

%

\begin{defn}
 An \emph{exhaustion} of $\G$ is a sequence $(S_i)_{i\in\mathbb{N}}$ of finite subsets $S_i\subseteq\G$ such that $S_i\subseteq S_{i+1}$ for every $i\in\mathbb{N}$
and $\bigcup_{i\in\mathbb{N}} S_i=\G$.
\end{defn}
The following criterion is easily verified and very useful in the applications:
\begin{lemma}\label{elementary:seminorm}
 Let $\alpha\in EH^{n+1}_b(\G)$, and let  
 $(S_i)_{i\in\mathbb{N}}$ be a fixed exhaustion of $\G$. Then for every sequence of elements $\varphi_i\in \QC^n(\G)^\G$, $i\in\mathbb{N}$, such that
\begin{enumerate}
  \item
$[\delta \varphi_i]=\alpha$  for every $i\in\mathbb{N}$,
\item
$\varphi_i|_{S_i}=0$ for every $i\in\mathbb{N}$,
\end{enumerate}
we have
$$\|\alpha\|_\bah\leq \liminf_{i\to \infty} \|\delta\varphi_i\|_\infty \ .$$
 Moreover, one can choose  elements $\varphi_i\in \QC^n(\G)^\G$, $i\in\mathbb{N}$, satisfying conditions~(1) and (2) in such a way that
 $$
 \|\alpha\|_\bah = \lim_{i\to \infty} \|\delta\varphi_i\|_\infty\ . $$
\end{lemma}

%

\subsection{Extension of quasi-cocycles from hyperbolically embedded subgroups}
Let us now suppose that $H$ is a hyperbolically embedded subgroup of $\G$, and recall
that 
$$
r^*\colon H^*_b(\G)\to H^*_b(H)
$$
is the restriction map induced by the inclusion of $H$ in $\G$.
We can now proceed with the proof of 
Theorem~\ref{FPS:thm}, which states that $r^{n+1}$ is an undistorted surjection for every $n\geq 1$, provided that we endow both
$EH^{n+1}_b(\G)$  and $EH^{n+1}_b(H)$ with the $\|\cdot\|_\bah$-seminorm (and that $H^n(H)$ is finite-dimensional).
The key ingredient for our argument  will be an extension result for quasi-cocycles
proved in~\cite{FPS}.

We first need to introduce the notion of \emph{small} simplex in $H$. Such notion depends on the geometry of the embedding of $H$ in $\G$. However,
for our purposes it is sufficient to know that 
 we can single out a particular finite subset $\overline{S}_0$ of $H$ with the property that an element $\overline{h}\in H^{n+1}$ is small
 if and only if $\overline{h}\in \overline{S}_0^{n+1}\subseteq H^{n+1}$ (see~\cite[Definition 4.7]{FPS}). In particular, the number of small simplices in $H$ is finite, so 
 for every cochain $\varphi\in C^n(H)$ 
 the finite number
$$
K(\varphi)=\max \{|\varphi (\ov h)|\, ,\  \ov h\subseteq H^{n+1}\ {\rm small}\}
$$
 is well defined.
 
 Now we have the following extension operator for quasi-cocycles on $H$:
 
 \begin{thm}[{\cite[Theorem 4.1]{FPS}}]\label{FPS:original:thm}
  There exists a linear map
 $$
 \Theta^n\colon  C_{\alt}^n(H)^{H}\to C_{\alt}^n(\G)^\G 
$$
such that  the following conditions hold for every $\varphi\in C_{\alt}^n(H)^{H}$:
\begin{enumerate}
\item\label{2cond}
$
\sup_{\overline{h}\in H^{n+1}} |\Theta^n(\varphi)(\overline{h})-\varphi(\overline{h})|\leq K(\varphi)\ ;
$\vspace{-.3cm}\\
\item\label{3cond}
if $n\geq 2$ then $
\|\Theta^n(\varphi)\|_\infty \leq n(n+1)\cdot \|\varphi\|_\infty\ ;
$\vspace{-.3cm}\\
\item\label{4cond}
$
\|\delta\Theta^n(\varphi)-\Theta^{n+1}(\delta\varphi)\|_\infty\leq 2(n+1)(n+2)K(\varphi)\ .
$
\end{enumerate}
 \end{thm}

 Theorem~\ref{FPS:thm} would readily follow from Theorem~\ref{FPS:original:thm} if we could get rid of the additive error $2(n+1)(n+2)K(\varphi)$ described in $(3)$ when estimating the defect
 of the extension in terms of the defect of the original quasi-cocycle. The following remark shows that this is not possible in general.
 
 \begin{rem}\label{distortion:qc}
 Let $M$ be an orientable complete finite-volume hyperbolic $3$-manifold with one cusp $C$. If $\G=\pi_1(M)$ and $H<\G$ is the subgroup corresponding to $C$, 
 then the pair $(\G,H)$ is relatively hyperbolic. In particular, $H$ is hyperbolically embedded in $\G$. 
Moreover, $H$ is isomorphic to $\mathbb{Z}\oplus\mathbb{Z}$, whence amenable, and the inclusion $C\hookrightarrow M$
induces a non-injective map $H_2(C)\to H_2(M)$ in homology. Therefore, \cite[Proposition 7.3]{FPS} shows that there is a genuine cocycle
$\varphi\in Z^2(H)^H$ such that $[\delta\Theta^n(\varphi)]$ is not null in $EH^{3}_b(\G)$. In particular, the defect of $\Theta^2(\varphi)$ cannot be zero,
whereas the defect of $\varphi$ vanishes. This shows that there cannot be any linear bound of the defect of $\Theta^2(\varphi)$ in terms of the defect of $\varphi$. 
 \end{rem}

 \subsection{Proof of Theorem~\ref{FPS:thm}}
 We are now ready to prove Theorem~\ref{FPS:thm}. Since in degree 2 the norm $\|\cdot\|_\bah$ is infinite on every non-trivial element
 (see Lemma~\ref{degree2:infinite}), the map $r^2\colon EH^2_b(G)\to EH^2_b(H)$ is obviously undistorted.
 Therefore, 
 Theorem~\ref{FPS:thm} follows from the following:
 
 \begin{thm}
  Let $n\geq 2$ and assume that $H^{n}(H)$ is finite dimensional. For every $\alpha\in EH^{n+1}_b(H)$ there exists $\beta\in\ EH^{n+1}_b(\G)$
  such that $r^{n+1}(\beta)=\alpha$ and
  $$
  \|\beta\|_\bah \leq n(n+1)\|\alpha\|_\bah\ .
  $$
 \end{thm}
 
 The difficulty in the proof arises from the fact that
 Theorem~\ref{FPS:original:thm} does not 
 actually give a map $EH^{n}_b(H)\to EH^n_b(\G)$, but just a map at the level of quasi-cocycles.
 
\begin{proof}
Of course if $\alpha=0$ there is nothing to prove,
so we may suppose $\alpha\neq 0$.

Let $(S_i')_{i\in\mathbb{N}}$ be any exhaustion of $\G$.
It readily follows from the construction of the map $\Theta^n$ in~\cite{FPS} that, for every $i\in\mathbb{N}$, there exists a finite subset $S_i''$ of $H$ such that
$\Theta^n(\varphi)$ vanishes on $S_i'$ whenever $\varphi$ vanishes on $(S''_i)^{n+1}$. Indeed, if $\overline{g}\in \G^{n+1}$, then 
$$
\Theta^n(\varphi)(\overline{g})=\sum_{B\in\mathcal{B}} \varphi'_B({\rm tr}_n^B(\overline{g}))\ ,
$$
where $B$ varies over all the left cosets of $H$ in $\G$, 
${\rm tr}_n^B(\overline{g})$ is a sort of weighted projection of $\overline{g}$ into $B$ (see \cite[Definition 4.5]{FPS}), 
and $\varphi'$ is obtained from $\varphi$ via the left translation by an element of $B$ (after setting $\varphi=0$ on small simplices contained in $H$).
It is proved in~\cite[Theorem 5.1]{FPS} that the sum in the above definition is finite (i.e.~$\varphi'_B({\rm tr}_n^B(\overline{g}))=0$ for all but a finite number of cosets $B$),
and it readily follows from the definition of ${\rm tr}_n^B$ that ${\rm tr}_n^B(\overline{g})$ is supported on a finite number of simplices in $B$ for every $\overline{g}\in\G^{n+1}$, $B\in\mathcal{B}$.
Therefore, we can choose as $S_i''$ the finite set given by the union of the translates in $H$ of the simplices in the support of ${\rm tr}_n^B(\overline{g})$, as 
$\overline{g}$ varies in $S_i'$ and $B$ varies among the cosets such that ${\rm tr}_n^B(\overline{g})\neq 0$.

Moreover, we can suppose that $S_i''\subseteq S''_{i+1}$ for every $i\in\mathbb{N}$ and that
$\bigcup_{i\in\mathbb{N}} S_i''=H$.
We now define an exhaustion $(S_i)_{i\in\mathbb{N}}$ of $H$ by setting
$$
S_i=S_i''\cup \overline{S}_0\, , \qquad i\in \mathbb{N}\ ,
$$
 (recall that $\overline{S}_0$ is the set of vertices of small simplices in $H$).

By Lemma~\ref{elementary:seminorm}, there exists a sequence $\varphi_i$ of invariant quasi-cocycles on $H$ 
such that the following conditions hold: for every $i\in\mathbb{N}$, $[\delta \varphi_i]=\alpha$ and 
$\varphi_i(\overline{s})=0$ for every $\overline{s}\in S_i^{n+1}$; moreover,
$$
\lim_{i\to\infty} \Def(\varphi_i)=\|\alpha\|_\bah\ .
$$
Since alternation does not increase the defect of a quasi-cocycle and does not alter the bounded class of its differential, we can also
assume that each $\varphi_i$ belongs to $\QCa^n(H)^H$.

Let us now set $\psi_i=\Theta^n(\varphi_i)\in C_{\alt}^n(\G)^\G$.
Since $\overline{S}_0\subseteq S_i$, for every $i\in\mathbb{N}$ we have $K(\varphi_i)=0$, so 
Theorem~\ref{FPS:original:thm} implies that
$$
\|\delta\psi_i\|_\infty=\|\delta \Theta^n(\varphi_i))\|_\infty =  \|\Theta^{n+1}(\delta\varphi_i)\|_\infty\leq n(n+1)\cdot \|\delta \varphi_i\|_\infty=n(n+1)\Def(\varphi_i)\ .
$$
In particular, $\psi_i$ is an alternating quasi-cocycle 
with $\Def(\psi_i)\leq n(n+1)\Def(\varphi_i)$. By construction we have $\psi_i|_{S'_i}=0$,  so the differential of $\psi_i$ defines a class $\beta_i\in EH^{n+1}_b(\G)$ such that
$$
\|\beta_i\|_{S_i'}\leq n(n+1)\Def(\varphi_i)\ .
$$
Moreover, since $K(\varphi_i)=0$, from Theorem~\ref{FPS:original:thm}~(1) we deduce that the restriction of $\psi_i$ to $H^{n+1}$ coincides
with $\varphi_i$, so that $r^{n+1}(\beta_i)=\alpha$ for every $i\in\mathbb{N}$. 

We are now going to prove the following:

\noindent{\bf Claim:}
The $\beta_i$ all belong to a finite-dimensional affine subspace of $EH^{n+1}_b(\G)$.

To this aim observe that, since $\delta\circ\delta=0$, for every element $c\in C^{n-1}_{\alt}(H)^H$ we have
$$
\delta(\Theta^n(\delta c))=\delta(\Theta^n(\delta c)-\delta\Theta^{n-1}(c))\ .
$$
Therefore, by Theorem~\ref{FPS:original:thm}~(3), the cochain $\delta(\Theta^n(\delta c))$ is the coboundary of an invariant bounded cochain, so that it defines
the trivial element of $EH^{n+1}_b(\G)$. In other words, the map
$$
Z^n_{\alt}(H)^H\to EH^{n+1}_b(\G)\, ,\qquad z\mapsto [\delta(\Theta^n(z))]
$$
induces a well-defined map
$$
j\colon H^n(H)\to EH^{n+1}_b(\G)\ .
$$


Since $H^n(H)$ is a finite-dimensional vector space, in order to prove the claim it suffices to
show that $\beta_i-\beta_j\in j(H^n(H))$ for every $i,j\in\mathbb{N}$.
Indeed, 
since $[\delta\varphi_i]=[\delta\varphi_j]$ in $EH^{n+1}_b(H)$, we have $\varphi_i-\varphi_j=b+z$, where $b\in C^n_{b,\alt}(H)^H$ is bounded and
$z\in Z^n_{\alt}(H)^H$ is a genuine cocycle. 
Since $n\geq 2$, by Theorem~\ref{FPS:original:thm}~(2) the cochain $\Theta^n(b)$ is bounded, hence 
$$
\beta_i-\beta_j=[\delta\Theta^n(\varphi_i)-\delta\Theta^n(\varphi_j)]=[\delta\Theta^n(b)+\delta\Theta^n(z)]=[\delta\Theta^n(z)]\in j(H^n(H))\ ,
$$
and this proves our claim. 

Let now $V$ be the subspace of $EH^{n+1}_b(\G)$ generated by the $\beta_i$, $i\in\mathbb{N}$, and let us set
$W=r^{n+1}(V)=\textrm{Span} \langle \alpha\rangle\subseteq EH^{n+1}_b(H)$. 
Assume first that there exists $\beta\in V\cap N_0^{n+1}(\G)$ such that
$r^{n+1}(\beta)\neq 0$. In this case $r^{n+1}(\beta)$ is a non-zero multiple of $\alpha$, so up to rescaling we may suppose $r^{n+1}(\beta)=\alpha$.
We have thus found an element $\beta$ in the preimage of $\alpha$ with vanishing $\|\cdot\|_\bah$-seminorm, and this certainly implies that $\|\beta\|_\bah\leq n(n+1)\|\alpha\|_\bah$, whence the conclusion.

We may then suppose that $V\cap N_0^{n+1}(\G)$ is contained in $\ker r^{n+1}$. In this case we first observe that, since $\|\cdot\|_{S'_i}\leq \|\cdot\|_{S'_{i+1}}$ and $\|\cdot\|_\bah =\sup_{i\in\mathbb{N}} \|\cdot\|_{S'_i}$,
the subspaces $M_i=\{\beta\in V\, |\, \|\beta\|_{S'_i}=0\}$  satisfy $M_i\supseteq M_{i+1}$ and $\bigcap_{i\in\mathbb{N}} M_i=V\cap N_0^{n+1}(\G)$. Since $V$ is finite-dimensional,
this implies in turn that there is $i_0\in\mathbb{N}$ such that $V\cap N^{n+1}_0(\G)=M_i$ for every $i\geq i_0$. Moreover, if we denote by $B_i$ the subspace of $V$ spanned by
$\{\beta_j,\, j\geq i\}$, then $B_i\neq \{0\}$  and $B_{i+1}\subseteq B_i$ for every $i\in\mathbb{N}$. Using again that $V$ is finite-dimensional, we obtain that, up to increasing $i_0$, we may suppose that
$B_i=B_{i+1}=B$ for every $i\geq i_0$. By definition, for every $i\geq i_0$ the subspace $B$ is spanned by elements with finite $\|\cdot\|_{S_i'}$-seminorm, and this implies that the seminorm $\|\cdot\|_{S_i'}$
is finite on $B$ for every $i\geq i_0$.


Let us now define the quotient space $\overline{V}=V/(V\cap N_0^{n+1}(\G))$ and let us set $\overline{B}=B/(B\cap N_0^{n+1}(\G))\subseteq \overline{V}$.
For every $i\geq i_0$, the seminorm $\|\cdot\|_{S_i'}$ induces a genuine finite and non-degenerate norm on $\overline{B}$, which will still be denoted by
$\|\cdot\|_{S_i'}$.
Let us denote by $\overline{\beta}_i$ the image of $\beta_i$ in $\overline V$. Then $\overline{\beta}_i\in \overline{B}$ for every $i\geq i_0$, and for every $j\geq i_0$ we have 
\begin{align}\label{final:estimate}
\limsup_{i\to \infty} \|\overline{\beta}_i\|_{S_j'} & \leq \limsup_{i\to \infty} \|{\beta}_i\|_{S'_j}
\leq n(n+1)\lim_{i\to \infty} \Def(\varphi_i)=n(n+1)\|\alpha\|_\bah\ .
\end{align}
In particular, the $\overline{\beta}_i$ are definitively contained in a bounded subset in the finite-dimensional normed space $(\overline{B},\|\cdot\|_{S_{i_0}'})$, and up to passing to a subsequence we can suppose that
$\lim_{i\to \infty} \overline{\beta}_i=\overline{\beta}$ in $(\overline{B},\|\cdot\|_{S_{i_0}'})$ for some $\overline{\beta}\in \overline{B}$. Observe now that, being genuine norms on the same finite-dimensional space $\overline{B}$,
the norms $\|\cdot\|_{S'_j}$, $j\geq i_0$, are all equivalent, so $\lim_{i\to \infty} \overline{\beta}_i=\overline{\beta}$ with respect to any norm $\|\cdot\|_{S'_j}$, $j\geq i_0$. Therefore, thanks to~\eqref{final:estimate}
we have
$$
\|\overline{\beta}\|_{S_j'}\leq n(n+1) \|\alpha\|_\bah
$$
for every $j\geq i_0$, hence for every $j\in\mathbb{N}$.

Let now $\beta\in V$ be any representative of $\overline \beta$. 
Using~\eqref{final:estimate} we may deduce that 
$$
\|{\beta}\|_{S_j'}= \|\overline{\beta}\|_{S_j'}\leq n(n+1)\|\alpha\|_\bah  
$$
for every $j\geq i_0$, and this implies in turn that $\|\beta\|_\bah\leq n(n+1)\|\alpha\|_\bah$ thanks to Lemma~\ref{elementary:seminorm}.
Therefore, in order to conclude it suffices to show that $r^{n+1}(\beta)=\alpha$.
By construction, the map $r^{n+1}$ induces a map $\overline{r}^{n+1}\colon \overline{V}\to W$ such that $\overline{r}^{n+1}(\overline{\beta})=r^{n+1}(\beta)$.
If we endow $W$ with its natural Euclidean topology (recall that $W$ is linearly isomorphic to $\R$), then the map $\overline{r}^{n+1}\colon \overline{V}\to W$, begin linear with a finite-dimensional domain, is continuous with respect
to any norm on $\overline{V}$.
We thus get
$$
r^{n+1}(\beta)=\overline{r}^{n+1}(\overline{\beta})=\overline{r}^{n+1}\left(\lim_{i\to \infty} \overline{\beta}_i\right)=\lim_{i\to\infty} {r}^{n+1}(\beta_i)=\lim_{i\to\infty} \alpha=\alpha\ .
$$
This concludes the proof.
 \end{proof}

\subsection{Zero-norm subspaces for acylindrically hyperbolic group}

We will now make use of the following fundamental result about acylindrically hyperbolic groups:
\begin{thm}[Theorem 2.24 of~\cite{DGO}]\label{F2:thm}
Let $\G$ be an acylindrically hyperbolic group. Then there exists a hyperbolically embedded subgroup $H$ of $\G$ such that
$H$ is isomorphic to $F_2\times K$, where $K$ is finite.
\end{thm}

The following result shows that Theorem \ref{generalcase} can now be reduced to
Theorem \ref{FPS:thm} and Theorem \ref{main:thm}:

\begin{cor}\label{reduction}
 Let $\G$ be an acylindrically hyperbolic group. Then  $\dim N_0^n(\G)\geq \dim N_0^n(F_2)$ for every $n\in\mathbb{N}$.
\end{cor}
\begin{proof}
Let $H$ be the hyperbolically embedded subgroup of $\G$ provided by Theorem~\ref{F2:thm}, and observe that $H$ surjects onto $F_2$ via an epimorphism
with finite (whence, amenable) kernel. By Corollary~\ref{amenable:inj} we have $\dim N_0^n(F_2)\leq \dim N_0^n(H)$, while Theorem~\ref{FPS:thm}
ensures that $\dim N_0^n(H)\leq \dim N_0^n(\G)$. The conclusion follows.
\end{proof}

\section{Relatively hyperbolic groups, cusped spaces and bicombings}\label{bicombing:sec}
In this section we collect some results about relatively hyperbolic groups that will be useful in the sequel. As described in the introduction, we are going to exhibit  non-trivial quasi-cocycles on
the free group $F_2$ by constructing a combinatorial version of (the primitive) of the volume form on a suitably chosen hyperbolic 3-manifold. By Milnor-Svarc Lemma, the fundamental group of any
closed hyperbolic $3$-manifold provides a discrete approximation of hyperbolic $3$-space. On the other hand, in order to get a quasi-isometric copy of hyperbolic $3$-space out of the fundamental group
$\G$ of a \emph{cusped} $3$-manifold we need to glue to the Cayley graph of $\G$ an equivariant collection of horoballs. We now briefly describe this procedure, closely following~\cite{GM}.

We will only consider simplicial graphs, i.e.~graphs without loops and without multiple edges between the same endpoints.
Every graph $G$ will be endowed with the path-metric $d_G$ induced by giving unitary length to every edge. The set of vertices of $G$ will be denoted by
$G^{(0)}$. Following~\cite[Definition 3.12]{GM}, we define the
\emph{(combinatorial) horoball} $\mathcal H(G)$ based on  $G$ as follows. The vertex set of $\mathcal H(G)$ is given by $G^{(0)} \times \N$, 
and two vertices $(g,n)$ and $(g',n')$ are joined by an edge if and only if one of the following conditions holds:
\begin{itemize}
    \item either $|n-n'|=1$ and $g=g'$, 
    \item or $n=n'$, $g \ne g'$ and $d_G(g, g') \le 2^n$.
\end{itemize}

Let us now fix a finitely generated group $\G$ with a distinguished\footnote{The constructions and the results that we are going to recall below also hold
for groups with a family of distinguished subgroups, but the case of a single subgroup is slightly easier and sufficient to our purposes.} finitely generated subgroup $H$. We choose a symmetric finite generating set $S$ for $\G$ containing a generating set for $H$,
and we denote by $\Cay(\Gamma,S)$ the associated Cayley graph, i.e.~the graph having $\G$ as the set of vertices, and such that two elements $g,g'\in \G=\Cay(\G,S)^{(0)}$ are joined by a single edge
if and only if $g^{-1}g'\in S$.  Observe that  the full subgraph of $\Cay(\Gamma,S)$ with vertices in $H$ coincides with the Cayley graph
of $H$ with respect to the generating set $S\cap H$. The left translation by $g\in \G$ induces an isomorphism between $\Cay(H,S\cap H)$ and the full subgraph of $\Cay(\Gamma,S)$ with vertices in $gH$,
which, in particular, is connected. We denote by  $\mathcal{H}_{gH}$ the combinatorial horoball based on such subgraph, and we identify the full subgraph of 
$\mathcal{H}_{gH}$ with vertices in $gH\times \{0\}$ with the full subgraph of $\Cay(\G,S)$ with vertices in $gH$. 

\begin{defn}[\cite{GM}]\label{defn:cusped} The \emph{cusped graph} $X$ associated to the pair $(\G,H)$ (and to a finite generating set $S$ as above) is the graph obtained 
by gluing a 
 combinatorial horoball $\mathcal{H}_{gH}$ to $\Cay(\G,S)$  for every left coset $gH$ of $H$,
 via the identification of the full subgraph of 
$\mathcal{H}_{gH}$ with vertices in $gH\times \{0\}$ with the full subgraph of $\Cay(\G,S)$ with vertices in $gH$. 
\end{defn}

\begin{rem}
 For ease of notation, we often do not distinguish between $X$ and its vertex set.
\end{rem}

The relative hyperbolicity of the pair $(\G,H)$ is encoded by the geometry of the cusped graph $X$ as follows:

\begin{thm}[{\cite[Theorem 3.25]{GM}}]
 The pair $(\G,H)$ is relatively hyperbolic if and only if the cusped graph $X$ is Gromov hyperbolic.
\end{thm}

\begin{rem} Indeed, there is a slight difference between our definition of cusped graph and Groves-Manning's one, in that our cusped graphs are necessarily simplicial, whereas
Groves and Manning explicitly allow multiple edges in their definition. We avoid double edges because it will be convenient to consider a cusped graph as contained in every Rips complex 
over it. However, in our applications we will be dealing only with torsion-free groups, for which our definitions precisely coincide with the ones in~\cite{GM}. 
\end{rem}

We have that $X^{(0)}$ is in canonical one-to-one correspondence with $\G\times\mathbb{N}$: this holds because we are dealing with the simple case of a pair $(\G,H)$ where $H$ is a single subgroup.
Henceforth we will tacitly make use of this identification, and denote vertices of $X$ by pairs in $\G\times\mathbb{N}$.
Following~\cite{GM}, we define the depth function
$$
D\colon X^{(0)}\to \mathbb{N}\, ,\qquad D(g,n)=n\ .
$$
For every horoball $\mathcal H\subseteq X$ and every $n\in\mathbb{N}$, we define  the \emph{$n$--horoball} $\mathcal{H}_n\subseteq X$ associated with 
$\mathcal H$ as the full subgraph of $X$ with vertices in  $ \mathcal H_n= D^{-1}([n,+\infty))\cap \mathcal{H}$. If a cusped graph $X$ is $\delta$--hyperbolic and $C > \delta$, then $C$--horoballs are convex in $X$  \cite[Lemma 3.26]{GM}.

\subsection{A quasi-geodesic bicombing}
We keep notation from the previous section, so that $X$ is the cusped graph associated to a relatively hyperbolic pair $(\G,H)$. 
Recall from Section~\ref{sec:2} that $C_i^{\Delta}(X)_\red$ denotes the space of reduced simplicial $i$-chains over $X$. 
A (homological) bicombing $Q\colon X\times X\to C_1^{\Delta}(X)_\red$ is a map such that
$$
\partial Q(x_0,x_1)=x_1-x_0\ .
$$
A bicombing is \emph{antisymmetric} if $Q(x_0,x_1)=-Q(x_1,x_0)$ for every $(x_0,x_1)\in X^2$, and \emph{$S$-quasi-geodesic} 
if there exists $S>0$ such that 
$\supp(Q(x_0,x_1))$ is contained in the $S$-neighborhood $\mathcal N_S(\gamma)$ of any geodesic $\gamma$ joining $x_0$ with $x_1$
(in~\cite{GM} there is the additional requirement that the norm of $Q(x_0, x_1)$ be bounded above by $S\cdot d(x_0,x_1)$; we will never need this in our argument).
 Moreover, $Q$ is \emph{equivariant} if $Q(g(x_0),g(x_1))=g\cdot Q(x_0,x_1)$ for every
$(x_0,x_1)\in X^2$, $g\in\G$. 

 Let $\sigma$ be a $1$-simplex in $X$
.  We define the maximal and the minimal depth of $\sigma$ as follows:
$\max D(\sigma)=\max \{D(v) ,\, v\in \supp(\sigma)\}$, while 
 $\min D(\sigma)=\min \{D(v) ,\, v\in \supp(\sigma)\}$ if there exists a horoball containing all the vertices of $\sigma$, and
$\min D(\sigma)=-\infty$ otherwise. Then
for any given chain $c\in C_1^{\Delta}(X)_\red$ we set
\begin{align*}
\max D(c) &= \max \{D(\sigma),\, \sigma\ {\rm appears\ in}\ c\}=\max \{D(v) ,\, v\in \supp(c)\}\in\mathbb{N},\\
\min D(c) &= \min \{D(\sigma),\, \sigma\ {\rm appears\ in}\ c\}\in\mathbb{N}\cup\{-\infty\}\ . 
\end{align*}

The existence of a quasi-geodesic bicombing with good filling properties, as stated in Theorem \ref{main theorem GM}, is essentially due to Groves and Manning (\cite[Section 5]{GM} and \cite[Theorem 6.10]{GM}). We fix the same notation as in~\cite{GM}, i.e.~we suppose $\delta\geq 1$ and we set
$$
K=10\delta, \quad L_1=100k,\quad
L_2=3L_1\ .
$$

As a quick guide, in the theorem below properties \ref{Q1},\ref{Q2},\ref{edge},\ref{2edge} are just convenient hypotheses to work with small simplices. Property \ref{sumzw}, combined with properties \ref{minDw} and \ref{maxDz}, says that a bicombing triangle split into a ``shallow'' part, $z$, and a part that lies deep into the horoballs, $w$. Both $w$ and $z$ are supported near a corresponding geodesic triangle by \ref{vicini}, and the shallow part has bounded norm by \ref{normT1} (implying that the bicombing chains cancel out nicely in the shallow part). Also, $z$ is alternating by \ref{antisymmetry}.

\begin{thm} \cite{GM} \label{main theorem GM} If $(\Gamma, H)$ is a relatively hyperbolic pair, then there are positive constants 
$T_1,S_1$,
an $S_1$-quasi-geodesic antisymmetric $\Gamma$-equivariant bicombing $Q$ on $X$, and $\G$-equivariant maps $z,w\colon X^3\to Z^{\Delta}_1(X)_\red$ with the following properties: \begin{enumerate}
    \item $Q(x_0,x_1) = [x_0, x_1]$, if $[x_0, x_1]$ is an edge of $X$;\label{Q1}
    \item if $x_0,x_1$ belong to the same $0$-horoball $\mathcal{H}$ and $d_\mathcal{H}(x_0,x_1)=2$, where $d_\mathcal{H}$ is the intrinsic distance on $\calH$, then
    $Q(x_0,x_1)=[x_0,y]+[y,x_1]$, where $y$ is a vertex in $\calH$; in particular,
    $\supp( Q({}x_0, {}x_1))\subseteq \mathcal H$;\label{Q2}
	\item $Q(x_0, x_1) + Q(x_1, x_2) + Q(x_2,x_0) = z(x_0,x_1,x_2)+w(x_0,x_1,x_2)$;\label{sumzw}
	\item  if $[x_i,x_j]$ is an edge of $X$ and $D(x_i)=0$ for every $i,j\in\{0,1,2\}$, then $z({}x_0, {}x_1, {}x_2)=[x_0,x_1]+[x_1,x_2]+[x_2,x_0]$ and $w(x_0,x_1,x_2)=0$; \label{edge}     
\item if $x_0,x_1,x_2$ belong to a 0-horoball $\mathcal H$ of $X$ and $d_\mathcal{H}(x_i,x_j)\leq 2$ for every $i,j\in\{0,1,2\}$,
        then $ \supp(z({}x_0, {}x_1, {}x_2))$ is contained in $\mathcal H\cap B(x_0,3)$, where $B(x_0,3)$ is the ball in $X$ of radius 3 centered at $x_0$;\label{2edge}
        	\item  $\supp (z(x_0, x_1, x_2)) \cup \supp(w(x_0,x_1,x_2))\subseteq \mathcal N_{S_1}(\gamma(x_0,x_1)\cup\gamma(x_1,x_2)\cup \gamma(x_2,x_0))$, 
	where $\gamma(x_i,x_j)$ is any geodesic joining $x_i$ with $x_j$;\label{vicini}
          	\item $\min D(w(x_0,x_1,x_2)) \ge L_2$;\label{minDw}
	\item $\max D (z(x_0,x_1,x_2)) \le 2L_2$;\label{maxDz}
	\item $\|z(x_0,x_1,x_2)\|_1 \le T_1$;\label{normT1}
	\item $z(x_{\tau(0)},x_{\tau(1)},x_{\tau(2)})={\rm sgn}(\tau)z(x_0,x_1,x_2)$ and $w(x_{\tau(0)},x_{\tau(1)},x_{\tau(2)})={\rm sgn}(\tau)w(x_0,x_1,x_2)$
	for every permutation $\tau$ of $\{0,1,2\}$.\label{antisymmetry}
\end{enumerate}
\end{thm}
\begin{proof}
We just define $Q(x,y)$ as the projection of $q_{x,y}$ in $C_1^{\Delta}(X)_\red$, 
where $q_{x,y}\in C_1^\Delta(X)$ is constructed in \cite{GM} as follows.
In \cite[Lemma 3.27]{GM} the authors choose an antisymmetric, $\G$-equivariant geodesic bicombing $\gamma$  with the property that if 
$x$ and $y$ lie in the same $L$-horoball, where $L>2\delta$, then $\gamma(x,y)$ 
consists of at most two vertical paths (of arbitrary length) and a horizontal path of length at most 3. Clearly if $d(x,y)=1$ then $\g(x,y)$ is the edge between $x$ and $y$.
Moreover, if $x$ and $y$ belong to the same $0$-horoball $\mathcal{H}$ and $d_\mathcal{H}(x,y)=2$, then 
it is readily seen that also $d(x_0,x_1)=2$, so that
the geodesic $\g(x,y)$ may be chosen to be equal to a concatenation $[x_0,y]*[y,x_1]$ for some vertex $y\in\mathcal{H}$.

For each pair of points $x,y$ Groves and Manning select an ordered subset $\calH_{x,y}$ of the set $\calC_{x,y}^K$  of $L_1$-horoballs  intersecting the $K$-neighborhood of $\g(x,y)$ \cite[Remark 4.2, Theorem 4.12]{GM},
and they define a \emph{preferred path} $p_{x,y}$ joining $x$ to $y$, in such a way that $p_{x,y}$ decomposes into the concatenation of minimizing geodesics between
the horoballs in $\calH_{x,y}$ and 
one suitably chosen path in each $\calH_{x,y}$~\cite[Definition 5.7]{GM}.

The homological bicombing $q_{x,y}$ is then obtained as follows~\cite[Definition 6.4]{GM}: one  decomposes $p_{x,y}$ as a concatenation of segments in $D^{-1}[0,L_2]$ and in $D^{-1}[L_2,\infty)$ 
where each segment in $D^{-1}[L_2,\infty)$ is contained in one element of $\calH_{x,y}$; then, each segment with endpoints $x_1,x_2$ in $D^{-1}[0,L_2]$ 
is replaced  by the antisymmetric bicombing $Q'_{x_1,x_2}$ constructed by Mineyev in \cite{Mineyev1}, and each segment contained in an $L_2$-horoball $\calH_{L_2}$, where $\calH\in\calH_{x,y}$, is replaced by
a path in the same $L_2$-horoball consisting of at most two vertical paths and a horizontal path of length 1 \cite[Definition 6.4]{GM}. 
Finally, $q_{x,y}$ is antisymmetrized.
Conditions~\eqref{Q1} and \eqref{Q2} now follow from the explicit description of $Q$ inside $L_2$-horoballs, together with the fact that
Mineyev's bicombing $Q'_{x_1,x_2}$ is obtained by antisymmetrizing $p_{x_1,x_2}=\gamma(x_1,x_2)$ if $d(x_1,x_2)\leq 10\delta$. 
Moreover, \cite[Proposition 6.5]{GM} implies that $Q$ is $S_1$-geodesic. 

Henceforth we denote by $Q$ also the obvious linear extension of $Q$ to linear combinations of pairs, so that 
$$
Q(\partial (x_0,x_1,x_2))=Q(x_0,x_1)+Q(x_1,x_2)+Q(x_2,x_0)\ .
$$

We first define the cycles $z(x_0,x_1,x_2)$ and $w(x_0,x_1,x_2)$ in the particular cases described in items~\eqref{edge} and~\eqref{2edge}.
We first suppose 
that  $[x_i,x_j]$ is an edge of $X$ and $D(x_i)=0$ for every $i,j\in\{0,1,2\}$. 
Then 
we set ${z}(x_0,x_1,x_2)=Q(\partial (x_0,x_1,x_2))=[x_0,x_1]+[x_1,x_2]+[x_2,x_3]$ and ${w}(x_0,x_1,x_2)=0$, and it is immediate to check that
this choice fulfills all the requirements of the statement.

Suppose now $x_0,x_1,x_2$ belong to a 0-horoball $\mathcal H$ of $\calX$ and $d_\mathcal{H}(x_i,x_j)\leq 2$ for every $i,j\in\{0,1,2\}$. By claims
\eqref{Q1} and~\eqref{Q2}, 
the cycle $Q(\partial (x_0,x_1,x_2))$ is the sum of at most 6 consecutive edges of $\mathcal{H}$, and $x_0$ is an endpoint of one of these edges. 
Therefore, the support of $Q(\partial (x_0,x_1,x_2))$ is contained in $\mathcal{H}\cap B(x_0,3)$. Let us now distinguish two cases:
if $D(x_i)\geq L_2+3$ for every $i=0,1,2$, then $\min D(\supp(Q(\partial (x_0,x_1,x_2))))\geq L_2$, and
we set
${z}(x_0,x_1,x_2)=0$, ${w}(x_0,x_1,x_2)=Q(\partial (x_0,x_1,x_2))$.
Otherwise, $\max D(\supp(Q(\partial (x_0,x_1,x_2))))\leq 2L_2$, and we set
${z}(x_0,x_1,x_2)=Q(\partial (x_0,x_1,x_2))$ and ${w}(x_0,x_1,x_2)=0$.

Let us now suppose that the triple $(x_0,x_1,x_2)$ does not fall into the cases described in items~\eqref{edge} and~\eqref{2edge}.
We denote by $\overline{z}(x_0,x_1,x_2)$ the reduced cycle associated to the cycle $c_{x_0x_1x_2}$ defined in~\cite[Definition 6.8]{GM},  and we set 
\begin{align*}
\overline{w}(x_0,x_1,x_2)&=Q(\partial (x_0,x_1,x_2))-\overline{z}(x_0,x_1,x_2)\ ,
\end{align*}
and
\begin{align*}
z(x_0,x_1,x_2)&=\frac{1}{6}\sum_{\tau\in \mathfrak{S}_3} \overline{z}(x_{\tau(0)},x_{\tau(1)},x_{\tau(2)})\ ,\\
w(x_0,x_1,x_2)&=\frac{1}{6}\sum_{\tau\in \mathfrak{S}_3} \overline{w}(x_{\tau(0)},x_{\tau(1)},x_{\tau(2)}). 
\end{align*}
Since $\overline{z}$ is a cycle, then so are $\overline{w}$, $z$ and $w$. Conditions \eqref{sumzw} and \eqref{antisymmetry} 
follow from the definitions and from the fact that $Q$ is antisymmetric. 

Since $z,w$ are obtained from $\overline{z},\overline{w}$ via alternation, in the proof of items \eqref{vicini}, \eqref{minDw}, \eqref{maxDz}, \eqref{normT1}
we can replace $z,w$ with $\overline{z},\overline{w}$, respectively.

The fact that $\overline{z},\overline{w}$ satisfy properties \eqref{minDw} and \eqref{normT1} is proved in \cite[Theorem 6.10]{GM}.

In order to show \eqref{vicini} and \eqref{maxDz} we need to describe the construction of $\overline{z}(x_0,x_1,x_2)$ in more detail. For any triple $(x_0,x_1,x_2)$ of vertices of $X$, a \emph{preferred
triangle} with vertices $x_0,x_1,x_2$ is a map $\psi\colon \partial \Delta^2\to X$ which takes the vertices and the sides of $\Delta^2$ respectively to $x_0,x_1,x_2$ and to the preferred paths 
$p_{x_0,x_1}$, $p_{x_1,x_2}$, $p_{x_2,x_0}$~\cite[Definition 5.28]{GM}. A skeletal filling of $\psi$ is a map $\ddot{\psi}\colon {\rm Skel}(\psi)\to X$ of $\psi$, 
where ${\rm Skel}(\psi)$ is a $1$-complex containing
suitable subdivisions of the sides of $\Delta^2$, and $\ddot{\psi}$ is a continuous map extending $\psi$ \cite[Definition 5.26]{GM}. A \emph{thick subpicture} of ${\rm Skel}(\psi)$
is (the quotient of) a subgraph of ${\rm Skel}(\psi)$ which is taken by (the map induced by) $\ddot{\psi}$ into the thick part $D^{-1}([0,L_2])$ of $X$ (and which enjoys several additional properties that we do not describe here,
see \cite[Definition 5.42]{GM}).
Finally, $c_{x_0,x_1,x_2}$ is a finite sum of terms of the form $Q(\ddot{\psi}(v),\ddot{\psi}(w))$, where $v,w$ are consecutive vertices of a thick subpicture of
${\rm Skel}(\psi)$ (see~\cite[Definition 6.8]{GM}).

Let us now prove \eqref{maxDz}. The explicit description of $\overline{z}(x_0,x_1,x_2)$ implies that, in order to bound  $\max D(\supp(\overline{z}(x_0,x_1,x_2)))=\max D(\supp(c_{x_0,x_1,x_2}))$, it is sufficient to bound 
$\max D(\supp(Q(\ddot{\psi}(v),\ddot{\psi}(w))))$, where $v,w$ are consecutive vertices of a thick subpicture. However, \cite[Proposition 5.43]{GM} implies
that, if $\gamma$ is any geodesic joining $\ddot{\psi}(v),\ddot{\psi}(w)$, then $\gamma$ does not intersect any $(L_1+L_2)$-horoball. Moreover, 
\cite[Proposition 6.5]{GM} implies that $\supp(Q(\ddot{\psi}(v),\ddot{\psi}(w)))$ is contained in the $(K+25\delta+9)$-neighborhood
of $\gamma$, so that $\max D(\supp(Q(\ddot{\psi}(v),\ddot{\psi}(w))))\leq L_1+L_2+K+25\delta+9\leq 2L_2$. This implies \eqref{maxDz}.

We are finally left to prove \eqref{vicini}. We first show that 
\begin{equation*}
\supp (\overline{z}(x_0, x_1, x_2)) \subseteq \mathcal N_{S_1}(\gamma(x_0,x_1)\cup \gamma(x_1,x_2)\cup\gamma(x_2,x_0))
\end{equation*}
for a suitably chosen universal constant $S_1$.
Indeed, as observed in the first paragraph of the proof of~\cite[Proposition 5.43]{GM}, if $v,w$ are two consecutive vertices of a thick subpicture of  ${\rm Skel}(\psi)$,
then either $d(\ddot{\psi}(v),\ddot{\psi}(w))=1$ or $v$, $w$ both lie on the same side of $\partial \Delta^2$. In the former case
$Q(\ddot{\psi}(v),\ddot{\psi}(w))=[v,w]$, which is supported in the $1$-neighborhood of $p_{x_0,x_1}\cup p_{x_1,x_2}\cup p_{x_2,x_0}$, which in turn is supported in
the $(K+12\delta+9)$-neighborhood of $\g(x_0,x_1)\cup \g(x_1,x_2)\cup \g(x_2,x_0)$ for any geodesic $\g(x_i,x_j)$ between $x_i$ and $x_j$ (see~\cite[Corollary 5.12]{GM}). In the latter case,
suppose that $\ddot{\psi}(u),\ddot{\psi}(v)$ lie on the preferred path $p_{x_i,x_j}$.
Let $p(\ddot{\psi}(v),\ddot{\psi}(w))$ be the subpath 
of $p_{x_i,x_j}$
with endpoints $\ddot{\psi}(v),\ddot{\psi}(w)$, and let $\gamma(\ddot{\psi}(v),\ddot{\psi}(w))$
be any geodesic with the same endpoints. Finally, let $\g(x_i,x_j)$ be any geodesic joining $x_i$ with $x_j$. 
By~\cite[Proposition~ 6.5]{GM}, the chain $Q(\ddot{\psi}(v),\ddot{\psi}(w))$ is supported in the
$(K+25\delta+9)$-neighborhood of $\gamma(\ddot{\psi}(v),\ddot{\psi}(w))$.
By~\cite[Corollary 5.13]{GM},
$p(\ddot{\psi}(v),\ddot{\psi}(w))$ is a quasi-geodesic with uniformly bounded quasi-geodesicity constants, so by hyperbolicity of $X$
there exists a universal constant $S'$ such that the Hausdorff distance between $\gamma(\ddot{\psi}(v),\ddot{\psi}(w))$ and $p(\ddot{\psi}(v),\ddot{\psi}(w))$ 
is bounded by $S'$. Finally, \cite[Corollary 5.12]{GM} ensures that $p_{x_i,x_j}$, whence $p(\ddot{\psi}(v),\ddot{\psi}(w))$, is contained in the $(K+12\delta+9)$-neighborhood
of $\g(x_i,x_j)$. Summing up, we have that the support of $Q(\ddot{\psi}(v),\ddot{\psi}(w))$ is contained in the $S_1$-neighborhood
of $\gamma(x_0,x_1)\cup \gamma(x_1,x_2)\cup\gamma(x_2,x_0)$, where $S_1=S'+2K+37\delta+18$. This concludes the proof that
\begin{equation*}
\supp (\overline{z}(x_0, x_1, x_2)) \subseteq \mathcal N_{S_1}(\gamma(x_0,x_1)\cup \gamma(x_1,x_2)\cup\gamma(x_2,x_0))\ .
\end{equation*}

Recall now from \cite[Proposition 6.5]{GM} that 
\begin{align*}
\supp (Q(\partial (x_0,x_1,x_2))) & \subseteq \mathcal{N}_{K+25\delta+9}(\gamma(x_0,x_1)\cup \gamma(x_1,x_2)\cup \gamma(x_2,x_0))\\ &\subseteq \mathcal{N}_{S_1}(\gamma(x_0,x_1)\cup \gamma(x_1,x_2)\cup \gamma(x_2,x_0))\  ,
\end{align*}
so from 
$$
\overline{w}(x_0,x_1,x_2)=Q(\partial (x_0,x_1,x_2))-\overline{z}(x_0,x_1,x_2)
$$
we now readily deduce that also
$$
\supp (\overline{w}(x_0, x_1, x_2)) \subseteq \mathcal N_{S_1}(\gamma(x_0,x_1)\cup \gamma(x_1,x_2)\cup\gamma(x_2,x_0))\ .
$$
This concludes the proof of item \eqref{vicini}.
\end{proof}

\subsection{Rips complexes on cusped graphs}
We are now interested in proving some results about fillings of cycles in relatively hyperbolic groups. It is well known that hyperbolic groups may be characterized
as those groups which satisfy a linear isoperimetric inequality, and an analogous characterization also holds for relatively hyperbolic groups, provided that
fillings are replaced by suitably defined \emph{relative} fillings. Classical isoperimetric inequalities usually deal with fillings of $1$-cycles via $2$-chains,
and in order to provide group-theoretic definitions of length and area it is usually sufficient to take generators and relations as unitary segments
and as  tiles of unitary area, respectively. However, in our argument we also need higher dimensional isoperimetric inequalities, which are better
stated in the context of higher dimensional
 complexes. To this aim it is often useful to consider Rips complexes (over augmented Cayley graphs, in our case of interest).

\begin{defn}\label{defn:rips} Given a graph $G$ and a parameter $1 \le \kappa \in \N$, the \emph{Rips complex $ \mathcal R_\kappa(G)$ on $G$} is the 
simplicial complex having $G^{(0)}$ as set of vertices, and an $n$--dimensional simplex for every $(n+1)$-tuple of vertices whose diameter in $G$ is at most $\kappa$.
\end{defn}

Let now $X$ be the cusped graph associated to the relative hyperbolic pair $(\G,H)$, as in the previous subsections. We fix a constant $\kappa\geq 4\delta+6$, where $\delta\in\mathbb{N}$ is a hyperbolicity constant for
$X$, and we set
$$
\calX=\calR_\kappa(X)\ .
$$

It is well known that, for $\kappa\geq 4\delta+6$, the Rips complex $\mathcal R_\kappa(G)$ of a $\delta$-hyperbolic graph is contractible (see for example \cite[3.$\Gamma$.3.23]{BriHae}). Therefore, we have the following:

\begin{prop} \label{i Rips sono contraibili} The simplicial complex $\calX$ is contractible.
\end{prop}

The notion of horoball easily carries over to $\calX$
as follows:

\begin{defn}
An ($n$--)horoball of $\calX$ is a full subcomplex of $\calX$ having the same vertices as an ($n$--)horoball of $X$.
\end{defn}

The maximal and the minimal depth of a chain 
$c\in C_n^{\Delta}(\mathcal{X})_\red$ are defined exactly as we did for $X$. 

Observe that, since $\kappa\geq 1$, the graph $X$ is naturally a subcomplex of $\calX$.
We stress the fact that, when we refer to the distance in $\calX$, we will always refer to the restriction of the distance of $X$ to the vertices of $\calX$: we will be never interested
in defining a metric on the internal part of $i$-simplices of $\calX$, $i\geq 1$, or in understanding the path metric associated to the structure of $\calX$ as a simplicial complex.
In particular, if $A$ is any subcomplex of $\calX$, then we denote by $\mathcal{N}_S(A)$ the full subcomplex of $\calX$ whose vertices lie at distance (in $X$) at most $S$ 
from the set of vertices of $A$.

The isometric action of $\G$ on $X$ induces a simplicial action of $\G$ on $\calX$.
As a consequence, each $C_n^{\Delta}(\calX)_\red$, $n\in\mathbb{N}$, is endowed with the structure of a \emph{normed $\G$-module} (i.e.~a normed space equipped with an isometric $\Gamma$--action).

The first author constructed in \cite{Franceschini2} fillings of cycles in $Z_k(\calX)$ with good properties:

\begin{thm} [{\cite[Theorem 5.6]{Franceschini2}}] \label{cicli in quasi grafi} Let $n,\,k,\,S \in \N$, $k \ge 1$.
Then there exists $S' = S'(n, k, S) \in \N$ such that, for every cycle $a \in Z_k^{\Delta}( \mathcal X)_\red$ and every family of geodesic segments 
$\alpha_1$, \ldots, $\alpha_n$ such that $\supp (a) \subseteq \mathcal N_S (\alpha_1 \cup \ldots \cup \alpha_n)$, there exists $b \in C_{k+1}^{\Delta}(\mathcal X)_\red$ with $ \partial b = a$ such that
\begin{enumerate}
\item $ \supp (b) \subseteq \mathcal{N}_{S'} (\supp (a))$ (in particular, $\supp(b)\subseteq \mathcal{N}_{S+S'}(\alpha_1\cup\ldots\cup\alpha_n)$),
\item $\|b\|_1 \le M(n, k, S,\max D(z)) \|a\|_1$,
\item if $\supp(a)$ is contained in a $(2\delta)$--horoball, then $\supp(b)$ is contained in the same  $(2\delta)$--horoball.\end{enumerate}
\end{thm}

\begin{defn}  Take $z \in C_{k}^{\Delta}( \mathcal X)_\red$. We say that a chain $a \in C_{k+1}^\Delta( \mathcal X)_\red$ is a \emph{relative filling} of $z$ if
$$
z=\partial a + c\ ,
$$
where $c$ is a chain in $ C_{k}^{\Delta}( \mathcal X)_\red$ with $\min D(c)\geq 0$ (i.e.~each simplex appearing in $c$ is contained in some $0$-horoball).
\end{defn}

We will now use Franceschini's  result to construct a relative filling of the bicombing $Q$ described above.

\begin{prop} \label{ofeuihwln} There exist constants $T_2,T_3 \in \R$ and 
a $\Gamma$--equivariant map $\varphi \: X^3 \to C_2^{\Delta}(\mathcal X)_\red$ such that,
for any triple $({}x_0, {}x_1, {}x_2)$ of vertices in $ X$: 
\begin{enumerate}
\item  if $[x_i,x_j]$ is an edge of $X$  and $D(x_i)=0$ for every $i,j\in\{0,1,2\}$, then $ \varphi({}x_0, {}x_1, {}x_2)=[x_0,x_1,x_2]$,         
\item if $x_0,x_1,x_2$ belong to a 0-horoball $\mathcal H$ of $\calX$ and $d_\mathcal{H}(x_i,x_j)\leq 2$ for every $i,j\in\{0,1,2\}$, 
        then $ \supp(\varphi({}x_0, {}x_1, {}x_2))\subseteq\mathcal H$,
        \item $\|\varphi ({}x_0, {}x_1, {}x_2)\|_1 \le T_2$,
 \item the chain $\varphi (\partial ({}x_0, {}x_1, {}x_2, {}x_3))$ admits a relative filling ${}B\in  C_{3}^{\Delta}( \mathcal X)_\red $ such that $\|B\|_1 \le T_3$.
\end{enumerate}
\end{prop}
\begin{proof} 
In order to get an equivariant map, we first
define $\varphi$ on a set of representatives for
the action of $\G$ on $X^3$, and then we extend $\varphi$ equivariantly.  Since $\G$ acts by isometries on $X$ and leaves the depth of points invariant, it is clear that this choice
is coherent with requirements (1) and (2) of the statement. 


Let us fix an element $(x_0,x_1,x_2)$ in the fixed set of representatives. 
We set
$$
z=z(x_0,x_1,x_2)\, ,\quad 
w=w(x_0,x_1,x_2)\ ,
$$
where $z(x_0,x_1,x_2)$ and $w(x_0,x_1,x_2)$ are the cycles provided by
 Theorem \ref{main theorem GM}. We will define  $\varphi(x_0,x_1,x_2)$ as a suitably chosen filling of $z$.

We first take care of the cases described in items (1) and (2). Suppose 
that  $[x_i,x_j]$ is an edge of $X$ and $D(x_i)=0$ for every $i,j\in\{0,1,2\}$. 
By Theorem~\ref{main theorem GM} \eqref{Q1}, 
we have $z=[x_0,x_1] + [x_1,x_2] + [x_2,x_0]$, and we just set $\varphi(x_0,x_1,x_2)=[x_0,x_1,x_2]$.

Suppose now $x_0,x_1,x_2$ belong to a 0-horoball $\mathcal H$ of $\calX$ and $d_\mathcal{H}(x_i,x_j)\leq 2$ for every $i,j\in\{0,1,2\}$. By claim \eqref{2edge} of Theorem~\ref{main theorem GM}, if $z=\sum_i \lambda_i [x'_i,x''_i]$, then $x_i',x_i''\in 
B(x_0,3)\cap\mathcal{H}$ for every $i$.
Since $\mathcal{X}=\mathcal{R}_\kappa (X)$ and $\kappa\geq 4$, 
this implies that 
$\{x_0,x_i',x_{i}''\}$ is the set of vertices of a simplex of $\calX$. Therefore, the sum $\varphi(x_0,x_1,x_2)=\sum_{i}\lambda_i [x_0,x_i',x_{i}'']$ 
defines an element in $C_2^{\Delta}(\mathcal{X})_\red$
supported on $\mathcal H$  such that $\partial \varphi(x_0,x_1,x_2)=z$ and $\|\varphi(x_0,x_1,x_2)\|_1=\|z\|_1\leq T_1$.

Suppose now that the triple $(x_0,x_1,x_2)$ does not satisfy the conditions described in items (1) and (2).
By Theorem \ref{main theorem GM} \eqref{vicini}, $\supp (z)$ is contained in the  
$S_1$--neighborhood of the union
$\gamma(x_0,x_1)\cup\gamma(x_1,x_2)\cup \gamma(x_2,x_0)$, where $\gamma(x_i,x_j)$ is any geodesic joining $x_i$ with $x_j$
(and $S_1$ 
does not depend on $(x_0, x_1, x_2)$). 
Moreover, $\max D(z)\leq 2L_2$ and $\|z\|_1\leq T_1$ 
because of \eqref{maxDz} and \eqref{normT1} of Theorem \ref{main theorem GM}. 
Hence by Theorem \ref{cicli in quasi grafi} there exists a chain $\varphi(x_0, x_1, x_2)\in C_2^{\Delta}(\mathcal{X})_\red$ 
such that $\partial\varphi(x_0,x_1,x_2)=z$ and $\|\varphi(x_0,x_1,x_2)\|_1\leq T_2$, where $T_2=M(3,1,S_1,2L_2) \cdot T_1$. This concludes
the proof of (1), (2) and (3).
Also observe that, if $S_1'=S'(3,1,S_1)+S_1$, then by Theorem~\ref{cicli in quasi grafi} (1) 
$$
\supp(
\varphi(x_0,x_1,x_2))\subseteq  \mathcal{N}_{S'_1}(\gamma(x_0,x_1)\cup\gamma(x_1,x_2)\cup \gamma(x_2,x_0))\ ,$$ where $\gamma(x_i,x_j)$ is any geodesic joining $x_i$ with $x_j$.

We now construct the relative filling $B$ of $\varphi (\partial ({}x_0, {}x_1, {}x_2, {}x_3))$ required to prove claim (4). Since $\varphi (\partial ({}x_0, {}x_1, {}x_2, {}x_3))$ is not a cycle, we need to find first a chain $c$ supported in the horoballs and satisfying $\partial c=\partial \varphi (\partial ({}x_0, {}x_1, {}x_2, {}x_3))$.  For the sake of conciseness, we will denote by $z$ and $w$ also the linear extensions of $z$ and $w$ over linear combinations of triples in $X^3$, so
that, for example, $z(\partial (x_0,\ldots,x_3))=\sum_{i=0}^3 (-1)^i z(x_0,\ldots,\widehat{x}_i,\ldots,x_3)$.
Let us fix $(x_0,\ldots,x_3)\in X^4$.
Since $Q\circ\partial=z+w$ and $Q\circ\partial\circ\partial =0$,  we have
$$
{z}(\partial (x_0,\ldots,x_3))=-{w}(\partial (x_0,\ldots,x_3))\ .
$$
Therefore, claims \eqref{minDw} and \eqref{maxDz} of Theorem~\ref{main theorem GM} imply that $$\max D({z}(\partial (x_0,\ldots,x_3)))\leq 2L_2\ ,$$  
$$\min D({z}(\partial (x_0,\ldots,x_3)))=\min D({w}(\partial (x_0,\ldots,x_3)))\geq L_2\ .$$ Moreover, by Theorem~\ref{main theorem GM} \eqref{normT1} we have
$$
\|{z}(\partial (x_0,\ldots,x_3))\|_1=\left\|\sum_{i=0}^3 (-1)^i {z}(x_0,\ldots,\widehat{x}_i,\ldots,x_3)\right\|_1\leq 4 T_1.
$$
Also observe that, by Theorem~\ref{main theorem GM} \eqref{vicini}, we have
$$
\supp({z}(\partial (x_0,\ldots,x_3))\subseteq \mathcal{N}_{S_1}\left(\bigcup_{i,j=0}^3 \gamma(x_i,x_j)\right)\ ,
$$
where $\gamma(x_i,x_j)$ is any fixed geodesic joining $x_i$ with $x_j$.
Since $L_2\geq 2\delta$, Theorem \ref{cicli in quasi grafi} implies that there exists a chain $c$ such that, if $T_1'=4M(6,1,S_1,2L_2)T_1$
and $S_1''=S'(6,1,S_1)+S_1$,
then
$$
\partial c =z(\partial (x_0,\ldots,x_3))\ ,
$$
$$
\supp(c)\subseteq \mathcal{N}_{S_1''}  \left(\bigcup_{i,j=0}^3 \gamma(x_i,x_j)\right)
$$
$$
\|c\|_1\leq M(6,1,S_1,2L_2)\cdot \|z(\partial (x_0,\ldots,x_3))\|_1\leq T_1'\ ,
$$
$$
\min D (c) \geq 2\delta
$$
(in particular, each simplex appearing in $c$ is contained in some horoball). 
Moreover, since $\supp(c)\subseteq \mathcal{N}_{S_1''}(\supp(z(\partial (x_0,\ldots,x_3))))$, 
we also have
$$
\max D(c)\leq \max D(z(\partial (x_0,\ldots,x_3)))+ S_1''\leq 2L_2+S_1''\ .
$$
Let us now consider the chain
$$
a=\varphi( \partial (x_0, x_1, x_2, x_3)) - c\ .
$$
By construction, $\partial a=0$, i.e.~$a$ is a cycle. If $S_1'''=\max \{S_1',S_1''\}$, then
$$\supp(a)\subseteq \mathcal{N}_{S_1'''}  \left(\bigcup_{i,j=0}^3 \gamma(x_i,x_j)\right)\ .$$ Moreover,
$$\max D(a)\leq \max \{D(c),D(z(\partial (x_0,\ldots,x_3)))\}\leq 2L_2+S_1''\ .$$
Let $B$ be the filling of $a$ provided by Theorem~\ref{cicli in quasi grafi}. By construction, 
$\partial B=a=\varphi (\partial (x_0, x_1, x_2, x_3)) - c$, so $B$ is a relative filling of $ \varphi(\partial (x_0, x_1, x_2, x_3))$. Moreover,
\begin{align*}
\|B\|_1& \leq M(6,2,S_1''',2L_2+S_1'')\|a\|_1\\ &\leq M(6,2,S_1''',2L_2+S_1'')\left(\|\varphi( \partial (x_0, x_1, x_2, x_3))\|_1+\|c\|_1\right)\\
& \leq M(6,2,S_1''',2L_2+S_1'')(4T_2+T_1')\ .
\end{align*}
This concludes the proof.
\end{proof}

\section{Combinatorial volume forms}\label{sec:4}
Before going into the proof of Theorem~\ref{volumeform:thm}, let us fix some notation.
Let $\G_0=F(a,b)$ be a free group of rank 2, and let $\psi\colon \G_0\to\G_0$ be a group automorphism induced by a pseudo-Anosov orientation-preserving  homeomorphism of a punctured torus.
Up to conjugating $\psi$, we may suppose that $\psi([a,b])=[a,b]$. Let $\G = \G_0 \rtimes_{\psi} \Z$, 
and denote by $t$ the generator of $\mathbb{Z}<\G$, in such a way that $ t g t^{-1}= \psi(g) $ for every $g\in \G_0$.
Let $H<\G$ be the subgroup generated by $[a,b]$ and $t$, and recall that the pair $(\G,H)$ is relatively hyperbolic. 
We denote by $X$ the cusped graph associated to the pair $(\G,H)$ and the generating set $S = \{a, b, [a,b], t\}$, and by $\calX$ the contractible Rips complex over $X$ defined in the previous section.
In fact, we will completely forget the structure of $X$ as a graph, and we will denote again by $X$ its set of vertices
(while we will make use of the structure of $\mathcal{X}$ as a simplicial complex).

Recall that $\G$ (hence, $\G_0$) acts freely on $X$, so the bounded cohomology of $\G_0$ may be isometrically computed via the complex 
$$C_b^n(\G_0\actson X):=\{\phi:X^{n+1}\to\R,\; \|\phi\|_\infty<\infty\}$$ introduced in Section~\ref{background:sec}. 
For every Lipschitz map $f\colon\mathbb{Z}\to\R$ we are going to construct a $2$-quasi-cocycle $\alpha_f\in \QCa^2({\G_0\actson X})^{\G_0}$. The quasi-cocycle $\alpha_f$   should be understood as a discrete approximation of a primitive of a volume form
on the infinite cyclic covering $M_0=\mathbb{H}^3/\G_0$ of the cusped hyperbolic manifold $M=\mathbb{H}^3/\G$ (here we are identifying $\G$ with its realization as a non-uniform lattice in the isometry group 
of $\mathbb{H}^3$).

Recall that $M_0$ is diffeomorphic to $\Sigma\times \R$, where $\Sigma$ is a once-punctured torus. 
If $\sigma$ is any $2$-simplex in $M_0$, then the evaluation on $\sigma$ of the primitive of a volume form on $M_0$ 
is equal to the volume of the prism spanned by $\sigma$ and by the projection of $\sigma$ on $\Sigma\times\{0\}$. Our construction
in inspired by this remark, yet it is
completely independent from the differential geometric situation just
recalled. We define a projection $p\colon X\to \G_0$ as follows: 
every element $X$ admits a unique expression as a pair $(g_0t^k,n)$ with $g_0\in\G_0$, and we then set
$$
p\colon X\to \G_0\, ,\qquad p(g_0t^k,n)=g_0\ .
$$

\subsection{Heuristic} In this subsection we just discuss the geometric meaning of $p$, the reader may safely skip ahead if the point is clear already.

The projection $p\colon X\to \G_0$ plays the role of the retraction of $M_0$ onto $\Sigma$. When considering $\G$ as a non-uniform hyperbolic lattice, the action of $t$ on
$\widetilde{M}\cong \widetilde{\Sigma}\times \R$ corresponds to the product of the lift of the pseudo-Anosov homeomorphism corresponding to $\psi$ (on $\widetilde{\Sigma}$) with the translation by $1$ (on $\R$).
Via the quasi-isometric identification between $X$ and $\widetilde{M}$, this action translates into the left action of $\G$ on $X$.
Observe now that the the group $\G$ acts on $X$ also on the right as follows:
$$
(g,n)\cdot g'=(gg',n)\ .
$$
This action is not by isometries, and it does not extend to a simplicial action over $X$ and neither over $\calX$. However, from the equality $ t g t^{-1}= \psi(g) $, $g\in\G_0$, we deduce that the right
action of $t$ on $X$ should correspond to the unitary translation on $\widetilde{\Sigma}\times \R$. Whence, the definition of $p$.
\subsection{The combinatorial area form}\label{subsec:area}
In order to compute (signed) volumes, we need
to introduce an orientation on triples in $X^3$. Let us fix a finite-area hyperbolization of $\G_0$, i.e.~a discrete faithful representation $\rho\colon \G_0\to \isom^+(\mathbb{H}^2)$ such that
$\mathbb{H}^2/\rho(\G_0)$ is isometric to a finite-volume once-punctured torus, and denote by $\cdot$ the action of $\G$ on $\partial\mathbb H^2$ induced by $\rho$.
We identify $\partial\mathbb{H}^2$ with the topological boundary of the Poincar\'e disk, and
we say that a triple of pairwise distinct points $(a_0,a_1,a_2)$ in $\partial\mathbb{H}^2$ is \emph{positive} (resp.~\emph{negative})
if $(a_0,a_1,a_2)$ are anti-clockwise (resp.~clockwise) oriented on $\partial \mathbb{H}^2$. Finally, if the points in the triple $(a_0,a_1,a_2)\in (\partial\mathbb{H}^2)^3$
are not pairwise distinct, we say that the triple is \emph{degenerate}.
The element $\rho([a,b])$ is parabolic, so it has a unique fixed point $\overline{q}\in \partial\mathbb{H}^2$.
We then define a map $\varepsilon\colon \G_0^3\to \{-1,0,1\}$ as follows:
$$
\varepsilon(g_0,g_1,g_2)=\left\{\begin{array}{rl} 
                                 1 & \textrm{if}\ (g_0\cdot\overline{q},g_1\cdot\overline{q},g_2\cdot\overline{q})\ \textrm{is\ positive}\\
                                 0 & \textrm{if}\ (g_0\cdot\overline{q},g_1\cdot\overline{q},g_2\cdot\overline{q})\ \textrm{is\ degenerate}\\
                                 -1 & \textrm{if}\ (g_0\cdot\overline{q},g_1\cdot\overline{q},g_2\cdot\overline{q})\ \textrm{is\ negative.}\\
                                \end{array}
\right. 
$$

We extend $\varepsilon$ to a map defined on $X^3$ by setting:
$$
\varepsilon (x_0,x_1,x_2)=\varepsilon (p(x_0),p(x_1),p(x_2))\ .
$$
The following result states that $\varepsilon$ is a $\G$-invariant bounded cocycle:

\begin{prop}\label{epsilon:prop}
 We have
 $$
 \varepsilon\in Z^2_b({\Gamma\actson X})^\G\ .
 $$
 Moreover, if $x_0,x_1,x_2$ all lie in a 0-horoball $\mathcal H$ of $\calX$, then $\varepsilon(x_0,x_1,x_2)=0$.
\end{prop}

\begin{proof}
 The fact that $\varepsilon$ is a cocycle is easily checked. In order to prove that $\varepsilon$ is $\G$-invariant it suffices to check that $g\cdot\varepsilon=\varepsilon$ for every
 $g\in\G_0$, and $t\cdot\varepsilon=\varepsilon$. 

 Let us fix $g\in\G_0$. It readily follows from the definition of $p$ that $p(gx)=gp(x)$ for every $x\in X$. Therefore, for every triple $(x_0,x_1,x_2)\in X^3$ we have
 $$
 \left(p(gx_0)\cdot\overline{q},p(gx_1)\cdot\overline{q},p(gx_2)\cdot\overline{q}\right)=
 \left(gp(x_0)\cdot\overline{q},gp(x_1)\cdot\overline{q}, gp(x_2)\cdot\overline{q}\right)\ ,
$$
and the conclusion easily follows from the fact that  $\rho(g)$ acts on $\partial\mathbb{H}^2$ as an orientation-preserving homeomorphism.

 In order to prove invariance with respect to $t$, first observe that the pseudo-Anosov homeomorphism $h\colon \Sigma\to\Sigma$ lifts to a quasi-isometry
 $\widetilde{h}\colon \mathbb{H}^2\to\mathbb{H}^2$ such that
 \begin{equation}\label{equivariance:eq}
 \widetilde{h}\circ \rho(g)=\rho(\psi(g))\circ\widetilde{h}
 \end{equation}
 for every $g\in\G_0$. The quasi-isometry $\widetilde{h}$ continuously extends to $\partial\mathbb{H}^2$, and equation~\eqref{equivariance:eq}
also holds when considering the actions of $\widetilde{h}$ and of $\G_0$ on $\partial\mathbb{H}^2$. In particular, if we set $g=[a,b]$ and we evaluate at $\overline q$
we obtain
 $$
\widetilde{h}(\overline{q})=\widetilde{h}(g\cdot\overline{q})=\psi(g)\cdot(\widetilde{h}(\overline{q}))=g\cdot(\widetilde{h}(\overline{q}))\ ,
 $$
 so $\widetilde{h}(\overline{q})$ is fixed by $g$, and
 $$
 \widetilde{h}(\overline{q})=\overline{q}\ .
 $$
 Now, for every $x=(g_0t^k,n)\in X$ with $g_0\in\G_0$, we have 
 $$
 p(tx)=p(t(g_0t^k,n))=p(tg_0t^k,n)=p(\psi(g_0)t^{k+1},n)=\psi(g_0)=\psi(p(x))\ ,
 $$
 hence
 \begin{equation} p(tx)\cdot\overline{q}=\psi(p(x))\cdot\overline{q} =  \psi(p(x))\cdot\widetilde{h}(\overline{q})=\widetilde{h}(p(x)\cdot\overline{q})\ .  \label{final invariance}  
\end{equation}
 Let us now consider a triple $(x_0,x_1,x_2)\in X^3$. 
 Observe that the trace of $\widetilde{h}$ on $\partial \mathbb{H}^2$ is an orientation-preserving homeomorphism. Therefore,
 the triple
 $
 (p(x_0)\cdot\overline{q}, p(x_1)\cdot\overline{q},p(x_2)\cdot\overline{q})
 $
 is positive (resp.~negative, degenerate) if and only if 
 $
( \widetilde{h}(p(x_0)\cdot\overline{q}), \widetilde{h}(p(x_1)\cdot\overline{q}), \widetilde{h}(p(x_2)\cdot\overline{q}))
 $
 is so. Thanks to~\eqref{final invariance}, this concludes the proof that $\varepsilon$ is $t$-invariant, whence $\G$-invariant.
 
Suppose now that $x_0,x_1,x_2$ all lie in the same horoball $\calH$ of $X$. Then $p(x_0),p(x_1),p(x_2)$ all lie in the same left coset of $\langle[a,b]\rangle$ in
$\G_0$. Using again that $\overline{q}$ is fixed by $\rho([a,b])$, this implies that 
$p(x_0)\cdot\overline{q}=p(x_1)\cdot\overline{q}=p(x_2)\cdot\overline{q}$, so
$\varepsilon(x_0,x_1,x_2)=0$.
\end{proof}

\subsection{The quasi-cocycle associated to a Lipschitz function}\label{sec:qc}
Let us now fix a Lipschitz function 
$$
f\colon \mathbb{Z}\to \R\ .
$$
We are going to define the quasi-cocycle $\alpha_f=\alpha(f)\in \QC_{{\rm alt}}^2(\G\actson X)^{\G_0}$ required in Theorem~\ref{volumeform:thm}.
The decomposition $\G = \G_0 \rtimes_{\psi} \Z$ of $\G$ as a semidirect product defines an epimorphism $\theta\colon \G\to\mathbb{Z}$
given by $\theta(g)=k$, where $g=g_0t^k$ is the unique expression of $g$ such that $g_0\in\G_0$. We extend $\theta$ to the whole of $X$
by setting $\theta(g,n)=\theta(g)$. 

We first define the simplicial cochain $F_f\in C^2_{\Delta,\alt}(\calX)$ such that,
if $\sigma$ is a 
$2$-simplex in $\calX$ with vertices $(x_0,x_1,x_2)\in X^3$, then
$$ \label{definition F_f}
F_f(\sigma) = \varepsilon(x_0,x_1,x_2) \frac{\sum_{i=0}^2 f(\theta(x_i))}{3}\ .  $$

\begin{lemma}\label{estimate:F}
We have $F_f\in C^2_{\Delta,\alt}(\calX)^{\G_0}$. Moreover:
\begin{enumerate}
\item 
$ F_f(\tau)=0$ for every $2$-simplex $\tau$ contained in a horoball,
\item 
$|F_f(\partial \sigma)|\leq R\cdot \lip(f)$ for every $3$-simplex $\sigma$ of $\calX$.
\end{enumerate}
\end{lemma}
\begin{proof}
The fact that $F_f$ is alternating (resp.~$\G_0$-invariant) follows from the fact that $\varepsilon$ is
(resp.~that $\varepsilon$ and $\theta$ are). Moreover, if the $2$-simplex $\tau=(x_0,x_1,x_2)$ is contained in a horoball, 
then Proposition~\ref{epsilon:prop} implies $\varepsilon(x_0,x_1,x_2)=0$, so $ F_f(\tau)=0$.

 Let now $(x_0,\ldots,x_3)$ be the vertices of a $3$-simplex $\sigma$. Recall that $\calX=\calR_\kappa(X)$,
 where $\kappa=4\delta+6$ and $\delta$ is a hyperbolicity constant for $X$. Since $x_i$ and $x_j$ are the vertices of a simplex in $\calX$ we have $d(x_i,x_j)\leq \kappa$. 
If $\max D(\sigma)\geq \kappa+1$, then $\sigma$ is contained in a horoball $\calH$, and hence $F_f(\partial \sigma)=0$. In particular we can assume $\max D(\sigma)\leq \kappa$.
 By definition, if $m=\sum_{j=0}^3 f(x_j)$, then 
 $$
 F_f(\partial_i\sigma)=\frac{ \varepsilon(x_0, \ldots, \widehat{x_i}, \ldots, x_3)}{3} \left(m-f(\theta(x_i))\right)\ ,
$$
so, using that
$\sum_{i=0}^3 (-1)^i \varepsilon({}x_0, \ldots, \widehat{{}x_i}, \ldots {}x_3) = 0$, we get:
\begin{align*}
F_f(\partial\sigma) & = \sum_{i=0}^3 (-1)^i \left(\frac{\varepsilon({}x_0, \ldots, \widehat{{}x_i}, \ldots, {}x_3)}3 \left(m- f(\theta({}x_i)) \right)\right) \\ &=  
 -\sum_{i=0}^3 (-1)^i \frac{\varepsilon({}x_0, \ldots, \widehat{x_i}, \ldots, {}x_3)}3 f(\theta({}x_i))\ .
\end{align*}
Using again that $\varepsilon$ is a cocycle with values in $\{-1,0,1\}$ we obtain that, up to a suitable permutation of  $({}x_0, {}x_1, {}x_2, {}x_3)$,
either $F_f(\partial\sigma)=0$, or $F_f(\partial\sigma)=f(\theta(x_2 )) -f( \theta( {}x_3))$, or
$F_f(\partial\sigma)=f(\theta( {x_0} )) - f(\theta( {x_1})) + f(\theta({x_2} )) - f(\theta( {x_3}))$. In any case, in order to conclude it is sufficient to show that
$|f(\theta(x_i))-f(\theta(x_j))|\leq R/2$ for a universal constant $R$ for every fixed pair of indices $i,j\in\{0,\ldots,3\}$. 

First observe that, being a homomorphism, the restriction of $\theta$ to $\G$ is $h$-Lipschitz for some $h>0$.
Recall that $d$ denotes the distance on $X$, and denote by $d_\G$ the distance on the Cayley graph of $\G$ with respect to the fixed generating set $S$.
If $x_i=(g_i,n_i)$, $x_j=(g_j,n_j)$, with $n_i,n_j\leq \kappa$, we claim that $d_\G(g_i,g_j)\leq \kappa 2^{2\kappa}$: indeed since $d(x_1,x_2)\leq \kappa$, any vertex in a geodesic in $X$ 
between $x_1$ and $x_2$ has depth at most $2\kappa$, hence a geodesic in $X$ between $x_1$ and $x_2$ projects to a path in $\G$ of length at most $\kappa 2^{2\kappa}$. 
The conclusion follows if we set $R/2=h\kappa 2^{2\kappa}$.
 \end{proof}
We are now ready to define the quasi-cocycle $\alpha_f$.  
For every triple $(x_0,x_1,x_2)\in X^3$ we set
$$ \alpha_f(x_0,x_1,x_2) = F_f(\varphi(x_0,x_1,x_2)), $$
where $\phi$ is the relative filling from Proposition \ref{ofeuihwln}.
\begin{prop}\label{quasi-cocycle:prop}
 We have
\begin{enumerate}
\item $\alpha_f\in \QC^2_{\alt}({\G_0\actson X})^{\G_0}$,
\item $\Def(\alpha_f)\leq K\cdot\lip(f)$ for a universal constant $K$. 
\end{enumerate}

\end{prop}
\begin{proof}
The $\G_0$-invariance of $\alpha_f$ follows from the $\G_0$-invariance of $\varphi$ (which, indeed, is even $\G$-invariant) and of $F_f$.
Moreover, $\alpha_f$ is alternating, since both $F_f$ and $\varphi$ are.

 In order to bound the defect of $\alpha_f$, let us fix a quadruple $(x_0,\ldots,x_3)\in X^4$, and estimate the value
\begin{align*}
\delta\alpha_f (x_0, x_1, x_2, x_3) & =\sum_{i=0}^3 (-1)^i \alpha_f(x_0, \ldots, \widehat{x_i}, \ldots, x_3)\\
&=\sum_{i=0}^3 (-1)^i F_f(\varphi(x_0, \ldots, \widehat{x_i}, \ldots, x_3))\\
&=  F_f \left(\sum_{i=0}^3 (-1)^i \varphi(x_0, \ldots, \widehat{x_i}, \ldots, x_3)\right). 
\end{align*}

Proposition~\ref{ofeuihwln} ensures the existence of a relative filling  of the chain $\sum_{i=0}^3 (-1)^i \varphi(x_0, \ldots, \widehat{x_i}, \ldots, x_3)$: we can choose a 3-chain $B\in C_{3}^\Delta(\calX)_\red$ with $\|B\|_1\leq T_3$ 
and such that the difference $\partial B -  \sum_{i=0}^3 (-1)^i \varphi(x_0, \ldots, \widehat{x_i}, \ldots, x_3)$ is a sum of simplices of $\calX$ contained in a union of  0-horoballs. 
We write ${}B = \sum_j \lambda_j \sigma_j$, with  $\sum_j |\lambda_j|\leq T_3$. 
  Since $F_f(\sigma) = 0$ if $\sigma$ is a simplex contained in a 0-horoball $\mathcal H$, we have:
\begin{align*}
 |\delta\alpha_f( (x_0, x_1, x_2, x_3))| &= \left|F_f  ( \partial B) -F_f \left(\partial B -  \sum_{i=0}^3 (-1)^i \varphi(x_0, \ldots, \widehat{x_i}, \ldots, x_3) \right) \right| \\
& = |F_f( \partial B)| = \left|\sum_j \lambda_j F_f( \partial \sigma_j)\right| \leq \sum_j |\lambda_j|\cdot |F_f(\partial \sigma_j)|  \\&\leq R\cdot \lip(f)\cdot \sum_j |\lambda_j| \leq
 R\cdot T_3\cdot \lip(f)\ ,
\end{align*}
 where $R$ is the constant provided by Lemma~\ref{estimate:F}. The conclusion follows.
 \end{proof}

\subsection{The $2$--cycle ${A}_m$ }
Purpose of this section is to construct, for each $m\in\N$, a cycle ${A}_m \in C_2(\G_0\actson X)_{\red,\G_0}$ on which we will evaluate our cocycles $\alpha_f$.

 In what follows  we will omit, for ease of notation, to distinguish a chain in $C_n(\G_0\actson X,\R)_{\G_0}$ from its reduced image in
 $C_n(\G_0\actson X,\R)_{\red,\G_0}$. For example, we will simply write
  $(x,y)=-(y,x)$ for every $(x,y)\in X^2$. 
We will construct the cycle ${A}_m$ as a union of different combinatorial analogues of geometric pieces. 
\begin{center}
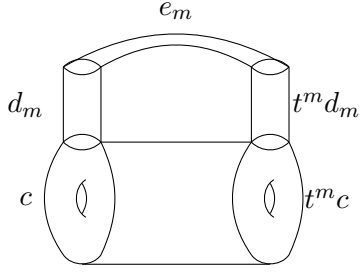
\begin{figure}[hh]
 \begin{tikzpicture}[scale=.5]
\node at (0,1.5) {$c$};
\draw (1,3) to [in=125, out=-125](1,0)to [in=-125, out=-55] (2,0) to [in=-65,out=65] (2,3) to [in=45, out=135](1,3);
\draw (2,3) to [in=-45, out=-135](1,3);
\node at (0,4) {$d_m$};
\draw (1,3) to (1,5);
\draw (2,3) to (2,5);
\draw (2,5) to [in=45, out=135](1,5);
\draw (2,5) to [in=-45, out=-135](1,5);
\draw (1.6,1) to [in=-150,out=150] (1.6,2);
\draw (1.5,1.1) to [in=-60,out=60] (1.5,1.9);
\node at (8,1.5) {$t^mc$};
\draw (6,3) to [in=125, out=-125](6,0)to [in=-125, out=-55] (7,0) to [in=-65,out=65] (7,3) to [in=45, out=135](6,3);
\draw (7,3) to [in=-45, out=-135](6,3);
\node at (8,4) {$t^md_m$};
\draw (6,3) to (6,5);
\draw (7,3) to (7,5);
\draw (7,5) to [in=45, out=135](6,5);
\draw (7,5) to [in=-45, out=-135](6,5);
\draw (6.6,1) to [in=-150,out=150] (6.6,2);
\draw (6.5,1.1) to [in=-60,out=60] (6.5,1.9);

\node at (4,6.5){$e_m$};
\draw (1,5) to [in=150, out=30] (7,5);
\draw (2,5) to [in=150, out=30] (6,5);
\draw(1.5,-.25) to (6.5,-.25);
\draw (2,3) to (6,3);
\end{tikzpicture}
\caption{The $2$--cycle ${A}_m$ }
\end{figure}
\end{center}
Let $e\in \G$ denote the identity element. The combinatorial analogue of a relative fundamental class for $\G_0\subseteq X$ is 
$$ c := \left((e, 0), (b, 0), (ba, 0)\right) +\left((e, 0), (ba, 0), (ab, 0)\right) + \left((e, 0), (ab, 0), (a,0)\right) .$$

An easy computation in $C_2(\G_0\actson X)_{\red,\G_0}$ gives
$$ \partial c = ((e, 0), ([a, b], 0))=( (ba, 0), (ab, 0)). $$

\begin{center}
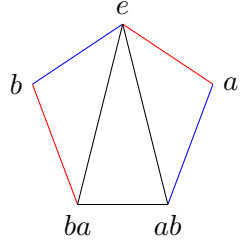
\begin{figure}
\begin{tikzpicture}[scale=.8]
\draw(2,3) [red] to (3.5,2);
\node at (2,3)[above] {$e$};
\draw (3.5,2)[blue] to (2.75,0);
\node at (3.5,2)[right]{$a$};
\draw (2.75,0) node [below]{$ab$} to (1.25,0);
\draw (1.25,0)[red]  to (0.5,2);
\node at (1.25,0)[below] {$ba$};
\draw (0.5,2)[blue] to(2,3);
\node at(0.5,2)[left]{$b$};
\draw (2,3) to (1.25,0) ;
\draw (2,3) to (2.75,0); 
\end{tikzpicture}
\caption{The relative fundamental class $c$}
\end{figure}
\end{center}

The second building block of our cycle ${A}_m$ is the combinatorial counterpart of a small annulus going deep enough into the horoball. We will need this to be able to join the boundaries of two combinatorial fundamental classes with two simplices.  
We will prove that, for each $K\in\N$, 
the 1-cycle $\partial c$ is homologous to the 1-cycle $a_K:=\frac 1{2^K}\left((e, K), ([a, b]^{2^K}, K)\right)$.  

We choose $K_m:=\lfloor \log_2 m\rfloor+1 $ big enough, so that $d((e, K_m), (t^m,K_m))=1$, and consider the chain

\begin{align*}d_m :=\sum_{i=0}^{K_m-1} \frac{1}{2^{i+1}} \Big(&\left((e, i), (e,i+1),( [a, b]^{2^i}, i)\right) + \\ + & \left(([a, b]^{2^{i}}, i), (e,i+1),([a, b]^{2^{i+1}}, i+1)\right) +\\+& \left(([a, b]^{2^{i}},i),([a, b]^{2^{i+1}}, i+1),([a, b]^{2^{i+1}},i)\right)\Big).
\end{align*}
In $C_2(X, \R)_{\red,\G_0}$, we have
$$ \partial d_m = -\partial c+a_{K_m}, $$
which, in particular, proves that $\partial c$ and $a_{K_m}$ are homologous.

\begin{center}
\begin{figure}[h]
\begin{tikzpicture}[scale=2]
\draw (0,0) node [below]{$\left(e,i\right)$} to (1,0) node [below]{$\left([a,b]^{2^{i}},i\right)$} to (2,0) node [below]{$\left([a,b]^{2^{i+1}},i\right)$} 
to (2,1)node [above]{$\left([a,b]^{2^{i+1}},i+1\right)$} to (0,1) node [above]{$\left(e,i+1\right)$} to (0,0);
\draw (0,1) to (1,0) to (2,1);
\draw (0,0) [red] to (0,1);
\draw (2,0) [red] to (2,1);

\end{tikzpicture}
\caption{A portion of the annulus $d_m$}
\end{figure}
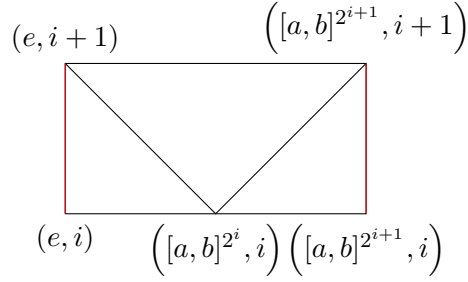
\end{center}

Our third and last building block is an annulus supported deep in the horoball with boundary $a_{K_m}-t^m a_{K_m}$. We will call it $e_m$:  
\begin{align*}
e_{m} := \frac{1}{2^{K_m}}&\Big(\left((e,{K_m}) , (t^m[a, b]^{2^{K_m}}\!,{K_m}), (t^m,{K_m})\right)\\
 + &\left((e, K), ([a, b]^{2^{K_m}},{K_m}),(t^m[a, b]^{2^{K_m}},{K_m})\right)\Big).
\end{align*}
In order to verify that $\partial e_m=a_{K_m}-t^ma_{K_m}$, we use that the pseudo-Anosov $\psi$ fixes the commutator $[a,b]$ and hence in particular $t^m[a,b]^{2^{K_m}}=[a,b]^{2^{K_m}}t^m$.
\begin{center}
\begin{figure}[h]
\begin{tikzpicture}[scale=1.7]
\node at (0,0) [below]{$\left(e,{K_m}\right)$};
\draw  (2,0) node [below]{$\left(t^m,{K_m}\right)$} 
to (2,1)node [above]{$\left(t^m[a,b]^{2^{{K_m}}},{K_m}\right)$} to (0,1) node [above]{$\left([a,b]^{2^{{K_m}}},{K_m}\right)$} to (0,0);
\draw (0,0) to (2,1);
\draw (0,0)[red] to (2,0);
\draw (0,1) [red]to (2,1);
\end{tikzpicture}
\caption{The annulus $e_m$}
\end{figure}
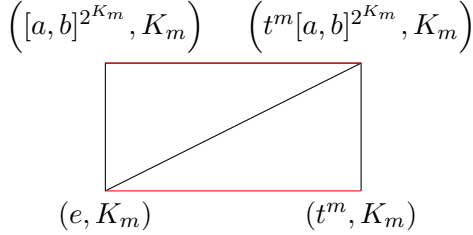
\end{center}

\noindent We can now define
$$ {A}_m := t^m\cdot (c + d_m) - (c + d_m) + e_m. $$ 
\begin{lemma}\label{lem:boundary}
Let $m\geq 0$. Then
\begin{enumerate}
 \item 
$\|A_m\|_1\leq 9$;
 \item 
the chain ${A}_m$ is a boundary in $C_2(\G_0\actson X)_{\red,\G_0}$.
\end{enumerate}
\end{lemma}
\begin{proof}
(1): We have
\begin{align*}
\|A_m\|_1&\leq 2\|c\|_1+2\|d_m\|_1+\|e_m\|_1\\ &\leq 6+3\sum_{i=0}^{K_m-1} 2^{-i-1}+ 2\cdot 2^{-K_m}\leq 6+3\sum_{i=0}^{\infty} 2^{-i-1}=9 \ .
\end{align*}
(2): We already verified that $\partial {A}_m=0$. Observe that each simplex involved in the definition of ${A}_m
$ exists in the Rips complex $\calX$ and hence ${A}_m$ also defines a cycle as a simplicial chain in $C_2^\Delta(\calX)_{\red,\G_0}=C_{2}^{\Delta}(\calX/\Gamma_0)_\red$. 
Since the Rips complex $\calX$ is contractible, the simplicial homology of $\calX/\G_0$ is canonically isomorphic to the homology of $\G_0\cong F_2$, which of course vanishes in degree 2.
Therefore, there exist a simplicial 3-chain ${B}_m\in C_3^\Delta(\calX)_{\red,\G_0}$ with $\partial {B}_m={A}_m$. The chain ${B}_m$ also defines an element of $C_3(\G_0\actson X)_{\red,\G_0}$ with $\partial {B}_m={A}_m$.
\end{proof}

\subsection{Proof of Theorem \ref{volumeform:thm}}

We now turn to the proof of Theorem \ref{volumeform:thm} that we recall for the reader's convenience:
\begin{thm}
 Let $\mathcal{L}(\mathbb{Z},\mathbb{R})$ be the space of Lipschitz real functions on $\mathbb{Z}$. There exist a constant $C>0$ and a linear map
 $$
 \alpha\colon \mathcal{L}(\mathbb{Z},\mathbb{R})\to \QCa^2({\G_0\actson X})^{\G_0}\,
 $$
 such that the following conditions hold:
\begin{enumerate}
 \item $\Def(\alpha(f))=\|\delta\alpha(f))\|_\infty \leq C\cdot \lip(f)$ for every $f\in\mathcal{L}(\mathbb{Z},\mathbb{R})$;
 \item $[\delta\alpha(f)]=0$ in $H^3_b(\G_0\actson X)\cong H^3_b(\G_0)$ if and only if $f$ is bounded.
\end{enumerate}
\end{thm}
Of course 
 the map 
$$\begin{array}{cccc}
 \alpha\colon &\mathcal{L}(\mathbb{Z},\mathbb{R})&\to &\QCa^2({\G_0\actson X})^{\G_0}\\
&f&\mapsto&\alpha_f
\end{array}$$
defined in Section \ref{sec:qc}
is linear, and we proved (1) in Proposition \ref{quasi-cocycle:prop}. The last missing step in the proof of Theorem \ref{volumeform:thm} is:
\begin{prop} \label{prop:final}
$\delta \alpha_f$ represents $0$ in $H_b^3(\G_0\actson X)$ if and only if the Lipschitz function $f$ is bounded. 
\end{prop}
In order to prove Proposition \ref{prop:final} we need to compute the value of $\alpha_f$ on ${A}_m$. We begin with a preliminary lemma:
\begin{lemma}\label{lem:epsilon(a,b,ab)}
Let $a,b$ be the generators of $\G_0$. We have $\varepsilon (e, ab, a)=\varepsilon(e, b, ba)=\pm1$ and $\varepsilon(e, ba, ab)=0$.
\end{lemma}
\begin{proof}
Recall that, in order to define $\varepsilon$, we chose a finite area hyperbolization $\rho$ of $\G_0$, and we denoted by $\ov q\in\partial \mathbb H^2$ the unique fixed point of $\rho([a,b])=\rho(a^{-1}b^{-1}ab)$. By definition, $\pi\varepsilon(g_0,g_1,g_2)$ is the area of the ideal triangle in $\mathbb H^2$ with vertices $(g_0\cdot \ov q,g_1\cdot \ov q,g_2\cdot \ov q)$.

Since $ba\cdot \ov q=ba\cdot( [a,b]\cdot \ov q)=ab\cdot \ov q$, the ideal triangle with vertices $(\ov q, ba\cdot \ov q, ab\cdot \ov q)$ is degenerate, and hence $\varepsilon(e, ba, ab)=0$. Moreover the union of the two triangles $ (\ov q, ab\cdot \ov q , a\cdot \ov q)$ and $ (\ov q, b\cdot \ov q , ba\cdot \ov q)$ is a fundamental domain for the $\G_0$ action on $\mathbb H^2$, and hence $\varepsilon (e, ab, a)=\varepsilon(e, b, ba)$ are non-zero. The common sign depends on the choice of the generators $a,b$.
\end{proof}

The next lemma shows that the cycle ${A}_m$ encloses a volume proportional to $m$:
\begin{lemma}\label{stimacrescita}
$|\alpha_f({A}_m)|=2|f(m)-f(0)|$.
\end{lemma}
\begin{proof}
We evaluate $\alpha_f({A}_m)$ term by term. Observe that every simplex in the support of ${A}_m$ has vertices at distance at most 2, and there exists a horoball $\calH$ such that 
$$\supp(d_m+e_m+t^md_m)\subset \calH.$$
It then follows from Proposition \ref{ofeuihwln} (2) that, for each simplex $\sigma$ in the support of $t^md_m+e_m-d_m$, we have $\phi(\sigma)\subset \calH$ and hence 
$$\alpha_f(t^md_m+e_m-d_m)=F_f(\phi(t^md_m+e_m-d_m))=0$$
by Lemma \ref{estimate:F} (1).
Therefore $\alpha_f({A}_m)=\alpha_f(t^mc-c)$.

We know from Lemma \ref{lem:epsilon(a,b,ab)} that $\varepsilon (e, ab, a)=\varepsilon(e, b, ba)=\pm1$ and $\varepsilon(e, ba, ab)=0$. Since $\varepsilon$ is $\G$-invariant (Proposition \ref{epsilon:prop}), and hence in particular $t$-invariant, we also deduce 
$\varepsilon (t^m, t^mab, t^ma)=\varepsilon(e, ab,a)$. Moreover, since all simplices involved in the definition of $c$ and $t^m c$ have vertices at distance at most 1, Proposition \ref{ofeuihwln} (1) implies that $\phi$ is the identity on each of them. 
We now have $\theta(v) = m$ for every vertex $v$ in $\supp(t^mc)$, and $\theta(v) = 0$ for every vertex $v$ in $\supp(c)$, so $\alpha_f(t^mc)-\alpha_f(c)=2\varepsilon(e, ab, a)(f(m)-f(0))$,
and this concludes the proof. 
\end{proof}

\begin{proof}[Proof of Proposition \ref{prop:final} ] Suppose that $f$ is bounded. For every $(x_0, x_1, x_2) \in X^3$, we write 
$$ \varphi(x_0, x_1, x_2) =  \sum_\sigma \lambda_\sigma \sigma$$ with $\sum_\sigma |\lambda_\sigma|\leq T_2$ (see Proposition \ref{ofeuihwln} (3)).
Then
\begin{align*} |\alpha_f(x_0, x_1, x_2)| &= |F_f(\varphi(x_0, x_1, x_2))|\\& \le\frac{1}{3} \sum_\sigma |\lambda_\sigma| \sum_{i=0}^2 |f(\theta(\sigma^{(i)}))| \\& \le T_2 \|f\|_\infty.\end{align*}
This shows that $\delta \alpha_f$ is the coboundary of a bounded $2$--cochain and hence represents 0 in $H^3_b(\G_0\actson X)$.

Vice versa, 
suppose that $f$ is a Lipschitz function such that $\delta\alpha_f = \delta \beta$ for some bounded $\G_0$-invariant cochain $\beta\in C^2_b(\G_0\actson X)^{\G_0}$.
By Lemma~\ref{lem:boundary}, there exists
a $3$--chain $B_m\in C_3(X,\R)_{\red,\G_0}$ with $ \partial {B}_m = {A}_m$.
Recall from Lemma~\ref{lem:boundary} that $\|A_m\|_1\leq 9$, so
\begin{align*}
|\alpha_f({A}_m)| & = |\alpha_f(\partial {B}_m
)| = |(\delta \alpha_f)({B}_m)| = |(\delta \beta)({B}_m)| \\ & = |\beta(\partial {B}_m)| = |\beta({A}_m)|\leq \|\beta\|_\infty\cdot \|A_m\|\leq  9\|\beta\|_\infty\ .
\end{align*}
Therefore, $|\alpha_f({A}_m)|$ is uniformly bounded. By Lemma~\ref{stimacrescita}, this implies that $|f(m)|\leq\left|\frac12\alpha_f({A}_m)\right|+|f(0)|$ is also uniformly bounded, i.e.~that $f$ is bounded,
as desired.
\end{proof}
\subsection{Proof of Theorem \ref{main:thm}}\label{sec:last}
In order to conclude the proof of Theorem~\ref{main:thm} we are now left to construct
 an uncountable set of linearly independent elements in $N^3_0(F_2)\subseteq H_b^3(F_2)$.
For every $f\in \calL(\Z,\R)$ we set 
$$
f_n(x) := \left\{ \begin{array}{ll}
f(x) - f(-n) & \textrm{ if } x \le -n \\
0 & \textrm{ if } -n \le x \le n \\
f(x) - f(n) & \textrm{ if } x \ge n \\
\end{array}
\right.
$$
and 
$$\calL^0(\Z, \R)= \left\{f\in \calL(\Z, \R)\, |\, \lim_{n\to \infty} \lip(f_n)=0\right\}\ .$$ 
We choose the basepoint $x=(e,0)\in X$ and we
set
$$
\eta\colon\calL (\Z, \R)\to H_b^3(\G_0)\, ,\quad \eta(f)=w^3_x([\delta \alpha_f])\ ,
$$
where $w^3_x\colon H^3_b(\G_0\actson X)\to H^3_b(\G_0)$ is the map described in Lemma~\ref{isometric:isom}.

Recall from Lemma~\ref{isometric:isom}
that the complex $C^*_{b,\alt}(\G_0\actson X)^{\G_0}$ isometrically computes the bounded cohomology of $\G_0$, so by Proposition~\ref{prop:final} the map 
$\eta$ induces an isomorphism between $\calL^0(\Z, \R)/(\calL^0(\Z,\R)\cap\ell^\infty(\Z))$ and $\eta(\calL^0(\Z, \R))\subseteq H_b^3(\G_0)$. 
It is immediate to realize that the dimension of the real vector space $\calL^0(\Z, \R)/(\calL^0(\Z,\R)\cap\ell^\infty(\Z))$ is equal to the cardinality of the continuum: for example, the classes of the maps
$n\mapsto n^\alpha$, $\alpha\in (0,1)$, define linear independent elements in $\calL^0(\Z, \R)/(\calL^0(\Z,\R)\cap\ell^\infty(\Z))$. Therefore, Theorem~\ref{main:thm}
is now reduced to the following:

\begin{prop}
For every function $f\in \calL^0(\Z, \R)$ we have $\eta(f)\in N_0^3(\G_0)$.
\end{prop}
\begin{proof}
For every $n\in\mathbb{N}$ we have $\|f - f_n\|_\infty=\max\left\{\left|f|_{[-n,n]}\right|\right\} < \infty$, hence by Proposition~\ref{prop:final} we have $\eta(f_n)=\eta(f)$ for every $n\in\mathbb{N}$.
Therefore, for every $n\in\mathbb{N}$ the cochain  $w_x^2(\alpha_{f_n})\in C^2(\G_0)$ is a primitive of $\eta(f)$.

Let us fix the exhaustion $(S_i)_{i\in\N}$ of $\G_0$ given by $$S_i=\{\g\in\G_0\, |\, d_{\G_0}(\g,e)\leq i\}.$$ 
For every $i$, we can choose $n_i$ big enough so that $w_x^2(\alpha_{f_{n_i}})|_{S_i^{3}}=0$: indeed the set $S_i$ is finite, and for each triple $(s_0,s_1,s_2)\in S_i^3$, 
the simplicial 2-chain $\phi((s_0,0),(s_1,0),(s_2,0))$ involves only a finite number of simplices. In particular we can find $n_i$ such that $|\theta(\supp\phi((s_0,0),(s_1,0),(s_2,0)))|\leq n_i$ for every 
$(s_0,s_1,s_2)\in S_i^3$, and for such $n_i$ we have $w_x^2(\alpha_{f_{n_i}})|_{S_i^{3}}=0$. Clearly we can also suppose that the sequence $\{n_i\}_{i\in\N}$ is monotonically diverging to $\infty$. 

Recall now that Proposition \ref{quasi-cocycle:prop} ensures that $\|\delta \alpha_{f_{n_i}}\|_\infty \le K \lip(f_{n_i})$, so since $f\in \calL^0(\Z,\R)$ we have
$\lim_{i\to\infty} \|\delta \alpha_{f_{n_i}}\|_\infty=0$.
Therefore, since $w_x^n$ is norm non-increasing, by Lemma \ref{elementary:seminorm} we finally get
$$\|\eta(f)\|_\bah\leq \liminf_{i\to\infty}\|\delta w_x^2(\alpha_{f_{n_i}})\|_\infty\leq\liminf_{i\to\infty}\|\delta \alpha_{f_{n_i}}\|_\infty 
=0\ .$$

\end{proof}
 \section{Appendix: volumes of mapping tori}
We use the techniques introduced in the paper to give a cohomological proof of (some particular cases) of an inequality due to Brock \cite[Theorem 1.1]{Brock}.

Let $\Sigma_g$ be the closed oriented surface of genus $g$, $g\geq 2$.
If $\psi\colon\Sigma_g\to\Sigma_g$ is a pseudo-Anosov homeomorphism, we denote by $M_\psi$ the mapping torus 
$$M_\psi:=\Sigma\times[0,1]/(x,0)\equiv(\psi(x),1).$$  
Recall that a pseudo-Anosov homeomorphism $\psi:\Sigma_g\to\Sigma_g$ is $\epsilon$-cobounded if the image of its Teichm\"uller axis in the moduli space $\cal M_g$ stays in the $\epsilon$-thick part $\cal M_g^\epsilon$ \cite[Section 2.1]{Farb-Mosher}.
 We denote by $\tau(\psi)$ the translation length of $\psi$ on the Teichm\"uller space endowed with the Teichm\"uller metric, which is well known to be equal to its maximal dilatation $\lambda$ \cite{Bers}.  For $\psi$ $\epsilon$-cobounded, it is 
 well known that the Teichm\"uller translation length is also uniformly bi-Lipschitz equivalent to the Weil-Petersson translation length. This can be seen as follows. Distances in both the Teichm\"uller and the Weil-Petersson metric can be computed, 
up to bounded multiplicative and additive error, in terms of the so-called subsurface projections to curve complexes of subsurfaces; this is known as the distance formula, see \cite[Theorem 6.12]{MM2}\cite[Theorem 4.4]{Brock-pants} for the Weil-Petersson case and \cite[Theorem 
1.1]{Rafi-DF} for the Teichm\"uller case. The difference between the distance formulas is that annular subsurfaces do not contribute in the Weil-Petersson case, while they do in the Teichm\"uller case. As observed in, e.g., \cite[Theorem 3.1]{KL-subgroups},
 it follows from \cite{Rafi-short} that in the $\epsilon$-cobounded case all subsurface projections to curve complexes of \emph{proper} subsurfaces are bounded, so that both in the Teichm\"uller and in the Weil-Petersson case the distance formula 
 only has one non-zero term, the one corresponding to the whole surface, easily implying the desired relation between translation distances.

The purpose of the appendix is to give a different proof of the lower bound of the volume of $M_\psi$ in terms of the dilation of $\psi$, when $\psi$ is an $\epsilon$-cobounded pseudo-Anosov. 

\begin{thm}\label{thm:A1}
There exists a constant $C>0$ depending only on $\epsilon$ and $g$ such that, for any $\epsilon$-cobounded pseudo-Anosov $\psi:\Sigma_g\to\Sigma_g$, we have
$$\vol(M_\psi)\geq C\tau(\psi).$$
\end{thm}

\begin{remark}
Brock proves \cite[Theorem 1.1]{Brock} that there is a constant $K$ depending only on the genus of the surface such that $$\frac{1}{K}\|\psi\|_{WP}\leq \vol(M_\psi)\leq K\|\psi\|_{WP}$$
for every pseudo-Anosov homeomorphism $\psi\colon \Sigma_g\to \Sigma_g$, where $\|\cdot \|_{WP}$ denotes the translation length of $\psi$ with respect to the Weil-Petersson metric. 
The lower bound is deduced  from the fact that in $M_\psi$ there are at least $\|\psi\|_{WP}$ short curves with disjoint Margulis tubes, each of which gives a definite contribution to the volume.

Upper bounds on the volume in terms of different translation lengths and with an explicit dependence on the genus are known: 
for any pseudo-Anosov $\psi\colon \S\to\S$,
Kojima and McShane \cite[Theorem 2 and Proposition 12]{KoMc} prove the inequality 
$$3\pi|\chi(\S)|\tau(\psi)\geq \vol(M_\psi)\ ,$$
while  Brock
 and Bromberg \cite{BrBr}  prove that $$\sqrt{3\pi/2(2g-2+n)}\|\psi\|_{WP}\geq \vol(M_\psi)\ .$$ 
\end{remark}

\subsection*{Proof summary}
Denote by $\G$ the fundamental group of $M_\psi$, and set $\G_0=\pi_1(\Sigma_g)$ so that $\G=\G_0\ast_{\psi_*}$, where $\psi_*$ denotes the automorphism of $\G_0$ induced by $\psi$. 
The strategy of our proof of Theorem \ref{thm:A1} is based on the ideas developed in the main paper: we construct an explicit combinatorial cocycle representing some multiple of the volume form of the three manifold $M_\psi$ and we compute its value on a suitable fundamental class.

In order to define our cocycle, we will first construct, as in Section \ref{sec:4}, a graph $X$ which is a discrete approximation of $\wt{M_\psi}$. As in the case of the graph considered in Section \ref{sec:4}, 
$X$ admits a $\G$-action, a $\G_0$-equivariant projection $p:X\to \G_0$, a $\G$-equivariant, 1-Lipschitz projection $\theta:X\to \R$. Furthermore $X$ is uniformly $\delta$ hyperbolic and has uniformly bounded degree. 
Therefore, as a consequence of Mineyev's Theorem, a suitable Rips complex $\calX$ over it admits a homological filling $\phi:X^3\to C^\Delta_2(\calX)$ with uniformly bounded norm, and uniformly bounded filling (Lemma \ref{lem:A2}).

Using the same ideas as in Section \ref{bicombing:sec} we use $\phi$ to construct a combinatorial primitive of the volume form: a $\G$-invariant quasi-cocycle $\alpha\in\QC^2(\Gamma\actson X)$. In Section \ref{sec:voles} we will use $\alpha$ to give a lower bound on the simplicial volume of $M_\psi$ and therefore on its hyperbolic volume.  

  \subsection{The graph $X$, a combinatorial approximation of $M_\psi$}
We assume (up to raising $\psi$ to a suitable power) that $\tau(\psi)$ is at least one. 
Let $\lambda$ be the sub-multiple of $\tau(\psi)$ in the interval $(0.5,1]$, set $k=\tau(\psi)/\lambda$. 

We denote by $l$ the Teichm\"uller axis of $\psi$ and choose a basepoint $0$ on $l\cong \R$. The group $\G$ acts on the canonical $\H^2$-bundle over $l$ \cite[page 107]{Farb-Mosher} and in particular, 
for each $s$, the subgroup $\G_0$ acts by isometries on the fiber at time $s$, that we denote by $\H^2_s$.  We also choose a basepoint $b_0$ on $\H^2_0$. The unique isometric lift of $l$ through $b_0$ allows us to choose coherent basepoints $b_s$ for each fiber.

\begin{lemma}
There is a constant $C$, depending on $g$ and $\epsilon$ only, such that for any $s\in\R$ we have
$${\rm diam}(\H^2_s/ \G_0)\leq C.$$
\end{lemma}
\begin{proof}
The homomorphism $\psi$ is $\epsilon$-cobounded, so every geodesic loop in $\Sigma_s=\mathbb H^2_s/\G_0$ has length at least $\epsilon$. 
This readily implies that the $\epsilon/4$-neighborhood of any length-minimizing geodesic of $\Sigma_s$ is isometric to the $\epsilon/4$-neighborhood of a geodesic of length $L$ in $\mathbb H^2$.

Since the area of the $\epsilon$-neighborhood of a geodesic of length $L$ grows linearly with $L$, and the area of $\Sigma_s$ is equal to $2\pi\chi(\Sigma)$, this provides the desired upper bound on
the lengths of minimizing geodesics in $\Sigma_s$, i.e.~on the diameter of $\H^2_s/ \G_0$.
\end{proof}

For each $n\in \Z$ denote by $X_n$ the graph whose vertex set is $\G_0\times \{n\}$ and with the property that two vertices $(g,n),(h,n)$ are joined by an edge if and only if $d(gb_{\lambda n},hb_{\lambda n})\leq 6C$. By Milnor-Svarc Lemma $X_n$ is $(6C,1)$-quasi-isometric to $\H^2_{\lambda n}$ (see, for example, the proof in  \cite[Proposition I.8.19]{BH}). Moreover $X_n$ has valency bounded above by $D(\epsilon,C)$.
The graph $X$ is the union of the $X_n$ with horizontal edges of type $((g,n),(g,n+1))$.

Observe that  there are a natural $\G$-action on $X$, a natural projection $p:X^{(0)}=\G_0\times \Z\to \G_0$ and a natural 1-Lipschitz map $\theta:X^{(0)}\to \R$ defined by $\theta(g,n)=n\lambda$. 

The graph $X$ is uniformly quasi-isometric to the canonical $\H^2$-bundle over $l$. Farb-Mosher \cite[page 145]{Farb-Mosher} use Bestvina-Feighn's combination theorem to show: 
\begin{prop}
There exists $\delta=\delta(g,\epsilon)$ such that $X$ is $\delta$-hyperbolic.
\end{prop}

Denote by $\calX$ the Rips complex over $X$ with constant $\kappa\geq 4\delta+6$. Since the graph $X$ is hyperbolic and has bounded valency, it admits a homological bicombing with a good filling $\phi$: 
\begin{lemma}\label{lem:A2}
There exist a constant $T_3\in\R$ depending on $\epsilon$ and $g$ only, and a $\G$-equivariant map $\phi:X^3\to C^\Delta_2(\calX)_\red$ such that
\begin{enumerate}
\item $\phi(x_0,x_1,x_2)=[x_0,x_1,x_2]$ if $d(x_i,x_i)\leq \kappa$;
\item for any 4-tuple $(x_0,x_1,x_2,x_3)$ of vertices of $X$ there exists $B\in C^\Delta_3(\calX)$ with $\phi(\partial(x_0,x_1,x_2,x_3))=\partial B$ and  $\|B\|_1\leq T_3$.
\end{enumerate} 
\end{lemma}
\begin{proof}
Since $X$ is $\delta$-hyperbolic and has uniformly bounded degree, Mineyev's construction gives a $\G$-equivariant,  anti-symmetric  homological bicombing $\phi_1:X^2\to C^\Delta_1(\calX)$ with the property 
that for each triple $(x_0,x_1,x_2)$, $\|\phi_1(\partial(x_0,x_1,x_2))\|_1\leq T_1(g,\epsilon)$ \cite[Theorem 10]{Mineyev1} (cfr. also \cite[Theorem 6.2]{GM} where it is observed that the constant $T_1$ in Mineyev's construction only depends on the valency of the 1-skeleton of $X$ and its hyperbolicity constant). 
By \cite[Proposition 12]{Mineyev1} there exists a filling $\phi:=\phi_2:X^3\to C^\Delta_2(\calX)_\red$ so that $\partial \phi(x_0,x_1,x_2)=\phi_1(x_0,x_1)+\phi_1(x_1,x_2)+\phi_1(x_2,x_0)$. Again by \cite{Mineyev1}, property (2) holds with a constant $T_3$ depending only on the hyperbolicity constant $\delta$ and the valency. We are free to modify $\phi$ on small simplices to ensure (1), because the norm of $\phi_1(x_0,x_1)$ is bounded by a function of the distance of $x_0,x_1$ (this is part of Mineyev's definition of 
quasi-geodesic bicombing).
\end{proof}
\subsection{The primitive of the volume form $\alpha$}
Just as we did in Section~\ref{sec:4},
in order to define a primitive $\alpha$ of a combinatorial volume form, we need to define a suitable sign $\epsilon$ for every triple of vertices of $X$. 

Fix a hyperbolization $\rho:\G_0\to\Isom(\H^2)$, and a lift $\wt h:\H^2\to \H^2$ of the pseudo-Anosov homeomorphism $\psi$ with a fixed point $\ov q$ in $\partial \H^2$. (Such a lift exists: consider a singular point $x\in\H^2$ for the lift of the singular foliation preserved by $\psi$. If we choose the lift $\wt h$  with the property that $\wt h(x)=x$, we get that the endpoints at infinity of the singular leaves through $x$ are fixed by $\wt h$.) 

For a triple $(x_0,x_1,x_2)$ of vertices of $X$, the sign $\epsilon(x_0,x_1,x_2)$ is $1$, $-1$ or $0$ depending on the orientation of the ideal triangle of $\mathbb H^2$ with vertices $\rho(p(x_i))\overline q$ 
(just as in Subsection \ref{subsec:area}). Proposition \ref{epsilon:prop} ensures that $\epsilon\in Z^2_b(\G\actson X)^\G$. 

Let $t\in \G$ be the stable letter of the HNN extension $\G=\G_0 *_{\psi_*}$.
As in Section \ref{sec:qc} we can define a simplicial cochain $F$ by setting
$$F([x_0,x_1,x_2])=\frac{1}3\epsilon(x_0,x_1,x_2)\sum_{i=0}^2 \theta(x_i).$$
In this context we have the following:
\begin{lemma}\label{lem:A3}
The simplicial cochain $F$ satisfies:
\begin{enumerate}
\item $|F(\partial[x_0,x_1,x_2,x_3])|\leq 2\kappa$ for any 3-simplex $[x_0,x_1,x_2,x_3]\in\calX$;
\item $\delta F$ is $\G$-invariant.
\end{enumerate}
\end{lemma}
\begin{proof}
The same computation as in Lemma \ref{estimate:F} gives that, up to reordering the vertices $x_i$, either $F(\partial [x_0,x_1,x_2,x_3])=0$ or $F(\partial[x_0,x_1,x_2,x_3])=\theta(x_2 ) - \theta( {}x_3)$, or
$F(\partial[x_0,x_1,x_2,x_3])=(\theta( {x_0} ) -\theta( {x_1})) + (\theta({x_2} ) -\theta( {x_3}))$. Since $\theta$ is 1-Lipschitz and $d(x_i,x_j)\leq \kappa$, (1) follows.

(2) follows from the description of $F(\partial [x_0,x_1,x_2,x_3])$ we have just given, and from the fact that  $\epsilon$ is $\G$-invariant and $\theta(t\cdot x)= \theta (x)+\tau(\psi)$.
\end{proof}

The $\G_0$-invariant primitive of the volume form is the evaluation of $F$ on fillings of simplices:
$$\alpha(x_0,x_1,x_2)=F(\phi(x_0,x_1,x_2)).$$
An immediate consequence of Lemma \ref{lem:A2} (2) and Lemma \ref{lem:A3} (1) is: 
\begin{lemma} 
The defect of $\alpha$ is uniformly bounded:
$$\|\delta\alpha\|_\infty\leq 2\kappa T_3\ .$$
\end{lemma}
\subsection{The volume estimate}\label{sec:voles} 
In order to estimate the simplicial volume, recall that the Rips complex $\calX$ is contractible, in particular we can choose a simplicial chain $S\in C_2^\Delta(\calX)_{\red,\G_0}$ representing the fundamental class of a fiber.
\begin{lemma} If $S=\sum c(\sigma)\sigma$ then
$$\sum_{\sigma}c(\sigma)\epsilon(\sigma)=-2\chi(\Sigma)\ .$$
\end{lemma}  
\begin{proof}
The oriented area of a hyperbolic triangle is $\pi$ times the orientation cocycle $\epsilon$. Using Gauss-Bonnet this implies that the pairing of $\epsilon$ with a fundamental class $S$ of the surface $\Sigma$ is 
equal to $-2|\chi(\Sigma)|$, that is the volume of the surface divided by $\pi$.
\end{proof}

\begin{lemma}
There exists a simplicial chain $M\in C_3(\G_0\actson X)_{\red,\G_0}$ with $\partial M= t\cdot S-S$. The image of $M$ in $ C_3(\G\actson X)_{\red,\G}$ represents the fundamental class $[M_\psi]\in H_3(M_\psi,\R)$.
\end{lemma}
\begin{proof}
Denote by $\ov M_\psi$ the infinite cyclic cover of $M_\psi$. Since $\calX$ is contractible,  the orbit maps define maps $C_*(\G_0)_{\red,\G_0}\to C_*(\G_0\actson X)_{\red,\G_0}$ 
and $C_*(\G_0)_{\red,\G_0}\to C_*(\ov M_\psi)_{\red}$ inducing isomorphisms in homology. 
Each complex is endowed with a $\Z$-action 
(with the positive generators of $\Z$ acting as $\psi_*$ on $C_*(\G_0)_{\red,\G_0}$, and as the positive generator of the deck transformation group of $\ov M_\psi$ on $C_*(\ov M_\psi)_{\red}$) 
and the isomorphisms are equivariant with respect to these actions. 

Using these isomorphisms the lemma follows from the corresponding statement for the topological counterpart, namely that if $[\S]\in H_2(\ov M_\psi,\R)$ is represented by a fiber and $\overline{t}$ 
is a generator of the deck group, then $\overline{t}[\S]-[\S]$ is the boundary of a 3-cycle projecting to the fundamental class of $M_\psi$.
\end{proof}

Recall that the simplicial volume $\|N\|$ of a closed oriented manifold $N$ is the $\ell^1$-seminorm of its real fundamental class. A fundamental result by Gromov and Thurston (see e.g.~\cite{Thurston}) states that there exists
a positive constant $v_n$ only depending on $n$ such that $\vol(N)=v_n \|N\|$ for every closed orientable hyperbolic $n$-manifold $N$. Therefore, Theorem \ref{thm:A1} is an immediate consequence of the following:

\begin{prop}
$$\|M_\psi\|\geq \frac {-2\tau(\psi)\chi(\Sigma_g)}{\kappa T_3}\ .$$
\end{prop}
\begin{proof}
Since $M$ represents $[M_\psi]$, we have
\begin{equation}\label{simplicial1}
|\delta\alpha(M)|\leq \|[M_\psi]\|_1\cdot \|[\delta\alpha]\|_\infty=\|M_\psi\| \cdot \|[\delta\alpha]\|_\infty \leq 2\kappa T_3 \|M_\psi\|\ . 
\end{equation}

By construction, for any $x\in X$ we have $\theta(t\cdot x)=\tau(\psi)+\theta(x)$. In particular, for any 3-simplex $\sigma=[x_0,x_1,x_2]\in\calX$  we get $\alpha(t\cdot \sigma)-\alpha(\sigma)=\epsilon(\sigma)\tau(\psi)$. This implies 
\begin{equation}\label{simplicial2}
 \delta\alpha(M)=\alpha(\partial M)=\tau(\psi)\sum_{\sigma}c(\sigma)\epsilon(\sigma)=-2\chi(\Sigma_g)\tau(\psi)\ .
\end{equation}
The conclusion now follows from~\eqref{simplicial1} and~\eqref{simplicial2}.
 \end{proof}
 

\end{document}